\documentclass[12pt,reqno]{amsart}
\usepackage{hyperref}
\usepackage{geometry}
\usepackage{enumitem}
\usepackage{amsmath}
\usepackage{amssymb}
\usepackage{amsthm}
\usepackage{amsrefs}
\usepackage{amsfonts}
\usepackage{xcolor}
\usepackage{mathtools}
\usepackage{stackengine}
\usepackage{verbatim}

\makeatletter
\@namedef{subjclassname@2020}{%
  \textup{2020} MSC}
\makeatother

\hypersetup{
	colorlinks   = true, %Colours links 
	urlcolor     = blue, %Colour for external hyperlinks
	linkcolor    = purple, %Colour of internal links
	citecolor   = blue %Colour of citations
}

\def\XXint#1#2#3{{\setbox0=\hbox{$#1{#2#3}{\int}$ }
\vcenter{\hbox{$#2#3$ }}\kern-.6\wd0}}

\makeatletter
\newcommand*{\rom}[1]{\expandafter\@slowromancap\romannumeral #1@}

\newcommand{\SL}{\text{SL}}
\newcommand{\GL}{\text{GL}}
\newcommand{\M}{\text{M}}

\newcommand{\X}{\mathcal{X}}
\newcommand{\Y}{\mathcal{Y}}
\newcommand{\R}{\mathbb{R}}

\newcommand{\T}{\mathbb{T}}

\newcommand{\e}{\epsilon}
\newcommand{\D}{\mathcal{D}}

\newcommand{\Oo}{\mathcal{O}}

\newcommand{\U}{\mathcal{U}}
\newcommand{\Z}{\mathbb{Z}}
\newcommand{\K}{\mathcal{K}}

\newcommand{\C}{\mathcal{C}}
\newcommand{\Cc}{\mathcal{C}_{c}^{\infty}}
\newcommand{\V}{\mathbf{v}}
\newcommand{\N}{\mathbb{N}}

\newcommand{\bthm}{\begin{thm}}
\newcommand{\ethm}{\end{thm}}
\newcommand{\bproof}{\begin{proof}}
\newcommand{\eproof}{\end{proof}}
\newcommand{\blem}{\begin{lem}}
\newcommand{\elem}{\end{lem}}
\newcommand{\brem}{\begin{rem}}
\newcommand{\erem}{\end{rem}}
\newcommand{\eeqn}{\end{equation}}
\newcommand{\eeqnn}{\end{equation*}}
\newcommand{\beqn}{\begin{equation}}
\newcommand{\beqnn}{\begin{equation*}}
\newcommand{\eprop}{\end{prop}}
\newcommand{\eexm}{\end{exm}}
\newcommand{\enexm}{\end{nexm}}
\newcommand{\ecor}{\end{cor}}
\newcommand{\bcor}{\begin{cor}}
\newcommand{\bexm}{\begin{exm}}
\newcommand{\bnexm}{\begin{nexm}}
\newcommand{\bprop}{\begin{prop}}
\newcommand{\bdefn}{\begin{defn}}
\newcommand{\edefn}{\end{defn}}
\newcommand{\benum}{\begin{enumerate}}
\newcommand{\eenum}{\end{enumerate}}

\title{Two Central limit theorems in Diophantine approximation}

\begin{document}
\theoremstyle{plain}
\newtheorem{thm}{Theorem}[section]
\newtheorem{lem}[thm]{Lemma}
\newtheorem{prop}[thm]{Proposition}
\newtheorem{cor}[thm]{Corollary}

\theoremstyle{definition}
\newtheorem{defn}[thm]{Definition}
\newtheorem{exm}[thm]{Example}
\newtheorem{nexm}[thm]{Non Example}
\newtheorem{prob}[thm]{Problem}

\theoremstyle{remark}
\newtheorem{rem}[thm]{Remark}

\author{Gaurav Aggarwal}
\address{\textbf{Gaurav Aggarwal} \\
School of Mathematics,
Tata Institute of Fundamental Research, Mumbai, India 400005}
\email{gaurav@math.tifr.res.in}

\author{Anish Ghosh}
\address{\textbf{Anish Ghosh} \\
School of Mathematics,
Tata Institute of Fundamental Research, Mumbai, India 400005}
\email{ghosh@math.tifr.res.in}

\date{}

\thanks{G.\ A.\ was an Infosys fellow during the writing of the paper.  A.\ G.\ gratefully acknowledges support from a grant from the Infosys foundation to the Infosys Chandrasekharan Random Geometry Centre. G. \ A.\ and  A.\ G.\ gratefully acknowledge a grant from the Department of Atomic Energy, Government of India, under project $12-R\&D-TFR-5.01-0500$. }


%\date{}

\begin{abstract} 
We prove central limit theorems for Diophantine approximations with congruence conditions and for inhomogeneous Diophantine approximations following the approach of Bj\"{o}rklund and Gorodnik. The main tools are the cumulant method and dynamics on homogeneous spaces.

\end{abstract}

\maketitle

\tableofcontents

\section{Introduction}
In this paper, we prove central limit theorems in Diophantine approximation, specifically for \emph{inhomogeneous} Diophantine approximation and for Diophantine approximation with \emph{congruence conditions}. Beginning with Schmidt's landmark paper \cite{S60} where he proved a quantitative analogue of Khintchine's theorem, there has been considerable activity in quantitative metric Diophantine approximation. These include pointwise results, namely variations of Schmidt's theorem in various settings, as well as more probabilistic ones, such as laws of large numbers and central limit theorems.  Recently, Bj\"{o}rklund and Gorodnik \cite{BG} used the method of cumulants and homogeneous dynamics, especially multiple mixing estimates \cite{BG3} for flows on homogeneous spaces to prove very general central limit theorems for Diophantine approximations. They treat systems of linear forms and also allow \emph{weights}. See also \cite{DFV} for an alternative approach to central limit theorems in Diophantine approximation. 

In this paper we study two variations of classical Diophantine approximation. Firstly, we consider the inhomogeneous setting; namely we consider small values of systems of linear forms accompanied by shifts, see (\ref{a1}). We next turn our attention to the problem of studying Diophantine inequalities where the rational approximations are constrained to satisfy congruence conditions see (\ref{c1}). We establish central limit theorems in both cases.

\section{Main results}
The first main result of our paper concerns inhomogeneous Diophantinbe approximation, specifically, weighted Diophantine approximation for systems of affine forms. Inhomogeneous approximation is an old and established line of study in Diophantine approximation. For qualitative results there are powerful transference theorems, see \cite{BL, BV, CGGMS} and the references therein, which often allow one to deduce inhomogeneous results from their homogeneous counterparts. This is not the case, as far as we are aware, for quantitative results, such as those in this paper.

Fix $m\geq 2$, $n \geq 1 $. Let $u = (u_{ij})\in M_{m \times n}(\mathbb{R})$ and $v= (v_j) \in M_{m\times 1}(\mathbb{R})$, and consider the family $(L_{u,v}^i)$ of affine linear forms defined by
\begin{equation*}
    L_{u,v}^i(x_1,x_2, \ldots, x_n) =\sum_{j=1}^n u_{ij}x_j +v_i,  \ \ \ i=1,\ldots,m
\end{equation*}
Fix a norm $\|.\|$ on $\R^n$. Fix $\vartheta_1, \vartheta_2, \ldots, \vartheta_m > 0$ and $w_1, w_2,\ldots, w_m > 0$ which satisfy 
\begin{equation*}
    w_1 + w_2+\ldots+w_m = n
\end{equation*}
and consider the Diophantine inequality 
\begin{equation}
\label{a1}
    |p_i + L_{u,v}^i(q_1, q_2, \ldots, q_n)| \leq \vartheta_i \|\bar{q}\|^{-w_i}, \ \ \ i=1,\ldots, m,
\end{equation}
with $(\bar{p},\bar{q})$= $(p_1,p_2,\ldots,p_m,q_1,q_2,\ldots,q_n) \in \mathbb{Z}^m \times (\mathbb{Z}^n \backslash \{0\})$. This represents a system of $m$ inhomogeneous linear forms in $n$ variables. Traditionally, the case $m=n=1$ has been the most well developed. In \cite{Szusz58}, Sz\"{u}sz proved a version of Khintchine's theorem in this setting.  

The number of solutions of this system with the norm of the ``denominator" $\bar{q}$ bounded by $T$ is defined by
\begin{equation}
    \label{defdel1}
    \Delta_T(u,v)= |\{(\bar{p},\bar{q}) \in \mathbb{Z}^m \times \mathbb{Z}^n: 0< \|\bar{q}\|<T \ \text{and \eqref{a1} holds} \}|.
\end{equation}

 An almost sure asymptotic formula for inhomogeneous simultaneous Diophantine approximation, i.e. $n = 1$ can be found in \cite{Harman}. In all the above results, the weights are taken to be equal.

Our first main result in this paper is a weighted central limit theorem in inhomogeneous Diophantine approximation.

\bthm \label{CLT2}
If $m \geq 2$, then for every $\xi \in \R$,
\begin{equation}
\left| \left\{ u \in M_{m \times n}([0,1]), v \in M_{m \times 1}([0,1]): \frac{\Delta_T(u,v)-C_{m,n} \log T}{(\log T)^{1/2}} < \xi \right\} \right| \longrightarrow Norm_{\sigma_{m,n}}(\xi)
\end{equation}
as $T \longrightarrow \infty$, where $$Norm_{\sigma}(\xi) = (2\pi \sigma )^{-1/2} \int_{- \infty}^{\xi} e^{-s^2/ (2 \sigma)} \, ds$$ denotes the normal distribution with variance $\sigma$ and
\begin{equation*}
    C_{m,n}=\sigma_{m,n}^2= 2^m \vartheta_1\vartheta_2 \ldots \vartheta_m \omega_n \text{ with } \omega_n := \int_{S^{n-1}}  d\bar{z},
\end{equation*}
where $S^{n-1}= \{ \bar{z}: \|\bar{z}\| = 1\}$. 
\ethm

\begin{rem}
Note that a similar Theorem still holds if we add the conditions in the definition of $\Delta_T(u,v)$, such as $$p_i + L_{u,v}^i(q_1, q_2, \ldots, q_n) \geq 0 \text{ or } < 0$$ and that $$q_j >0 \text{ or } \leq 0.$$ In general, let $S$ denote a non-empty union of orthants in $\R^{m+n}$ (all of whose adjacent sides may or may not be included), and define
\begin{equation*}
     \Delta_T^S(u,v)= |\{(\bar{p},\bar{q}) \in \mathbb{Z}^m \times \mathbb{Z}^n: 0< \|\bar{q}\|<T \ \text{and \eqref{a1} holds, } (\bar{p}+ u.\bar{q} + v, \bar{q}  ) \in S \}|,
\end{equation*}
then the proof of Theorem \ref{CLT2} still works to give a similar theorem with a different mean and variance (which also vary, according to the sides of $S$, which are included in $S$). 
\end{rem}

This Theorem is proved for equal weights by Dolgopyat, Fayad and Vinogradov in \cite{DFV}, and it is plausible that the approach in that paper also suffices to prove the weighted version above.

Our second result concerns Diophantine approximation with congruence conditions. For $u \in M_{m \times n}(\mathbb{R})$, consider the family $(L_{u}^i)$ of linear forms defined by
\begin{equation*}
    L_{u}^i(x_1,x_2, \ldots, x_n) =\sum_{j=1}^n u_{ij}x_j, \ \ \  i=1,\ldots,m
\end{equation*}
Fix a norm $\|.\|$ on $\R^n$ and a natural number $N \geq 1$ and vector $v \in \Z^{m+n}$ such that $\gcd(v,N)=1$. Fix $\vartheta_1, \vartheta_2, \ldots, \vartheta_m > 0$ and $w_1, w_2,\ldots, w_m > 0$ which satisfy 
\begin{equation*}
    w_1 + w_2+\ldots+w_m = n
\end{equation*}
and consider the Diophantine inequality
\begin{equation}
\label{c1}
    |p_i + L_{u}^i(q_1, q_2, \ldots, q_n)| \leq \vartheta_i \|\bar{q}\|^{-w_i}, \ \ (\bar{p},\bar{q})= v \ (mod \ N) \ \ \ i=1,\ldots, m,
\end{equation}
with $(\bar{p},\bar{q})$= $(p_1,p_2,\ldots,p_m,q_1,q_2,\ldots,q_n) \in \mathbb{Z}^m \times (\mathbb{Z}^n \backslash \{0\})$. The number of solutions of this system with the norm of the ``denominator" $\bar{q}$ bounded by $T$ is defined by
\begin{equation}
    \Delta_{T,v,N}(u)= |\{(\bar{p},\bar{q}) \in \mathbb{Z}^m \times \mathbb{Z}^n: 0< \|\bar{q}\|<T \ \text{and \eqref{c1} holds} \}|.
\end{equation}

This problem also has an old vintage. A Khintchine type theorem for the case $m=n=1$ was proved by Hartman and Sz\"{u}sz in \cite{HS}. Subsequently, Szusz obtained a quantitative strengthening in \cite{Szusz}, again for $m = n = 1$.  More recently, and in connection to systems of linear forms, a Khintchine-Groshev type Theorem was proved in \cite{NRS}. In his Mathscinet review of \cite{Szusz}, Cassels remarked that it would be interesting to obtain a simultaneous version of Szusz's quantitative result. This was achieved in \cite{AGY} where an almost sure pointwise bound on $\Delta_{T,v,N}(u)$ was obtained. The latter two papers, i.e. \cite{NRS, AGY} use homogeneous dynamics. 

Our second main result in this paper is a weighted central limit theorem for Diophantine approximation with congruence constraints.

\bthm \label{CLT3}
If $m \geq 2$, then for every $\xi \in \R$,
\begin{equation}
\left| \left\{ u \in M_{m \times n}([0,1]): \frac{\Delta_{T,v,N}(u)-C_{m,n,N} \log T}{(\log T)^{1/2}} < \xi \right\} \right| \longrightarrow Norm_{\sigma_{m,n,N}}(\xi)
\end{equation}
as $T \longrightarrow \infty$, where
\begin{equation*}
    C_{m,n,N}=C_{m,n}/N^{m+n}
\end{equation*}
and
\begin{equation*}
\sigma_{m,n,N} = \frac{2^{m+1}}{N^{n+m}}\left( \prod_{i=1}^{m} \vartheta_i\right) \omega_n \left(  1 + \frac{2}{\zeta_N(m+n)} \sum_{r \in S_N} \sum_{q \geq 1} \frac{q-1}{(Nq+r)^{m+n}}\right),
\end{equation*}
where  $S_N=\{ i \in \Z : 0 \leq i < N, \gcd(i,N)= 1\}$ and $\zeta_N(r):= \sum_{\substack{ k \geq 1 \\ \gcd(k,N)=1}} k^{-r}$.
\ethm

\begin{rem}
    Note that for $N=1$, $S_N=\{0\}$ and $\zeta_N$ is the usual Riemann's $\zeta$-function and the result, in this case, is proved in \cite{BG}. Also, the case $\gcd(v,N) \neq 1$ is easy to derive from Theorem \ref{CLT3}. 
\end{rem}

\begin{rem}
As in Theorem \ref{CLT2}, a similar Theorem still holds for
\begin{equation*}
    \Delta_{T,v,N}^S(u)= |\{(\bar{p},\bar{q}) \in \mathbb{Z}^m \times \mathbb{Z}^n: 0< \|\bar{q}\|<T \ \text{and \eqref{c1} holds}, (\bar{p}+ u.\bar{q}, \bar{q}  ) \in S  \}|,
\end{equation*}
where $S$ denotes a non-empty union of orthants in $\R^{m+n}$ (all of whose adjacent sides may or may not be included). Here again, the proof of Theorem \ref{CLT3} still works to give a similar theorem with a different mean and variance (which also vary, according to the sides of $S$, which are included in $S$ ).  
\end{rem}

\begin{rem}
    The inhomogeneous Diophantine approximation studied in this paper is sometimes called ``doubly inhomogeneous" in the number theory literature. The reason is that both parameters, the homogeneous form as well as the shift, are allowed to vary. There is also ``singly inhomogeneous" approximation, where one of the parameters, either the homogeneous form or the shift, is fixed. This latter study is related to Diophantine approximation with congruence constraints; namely if the shift ($v_i$ in our notation) in the inhomogeneous form is a fixed rational vector, then the inhomogeneous problem reduces to the congruence one. This observation is used in the context of quadratic forms in \cite{GKY}.
\end{rem}

\subsection{Acknowledgements} Part of this work was done at the ICTS, Bengaluru during the program ``Ergodic Theory and Dynamical Systems". We thank them for excellent working conditions. A.\ G.\ thanks Barak Weiss for many helpful and enjoyable conversations about Diophantine approximation with congruence constraints.

\section{Proof of the inhomogeneous case}  \label{Affine}

\subsection{Introduction} \label{Aff Intro}  \hfill \\

We will follow the approach in the paper  \cite{BG} to prove our main theorems. For this section, we fix $m,n \in \N$ with $m \geq 2$ and $\|.\|$, a norm on $\R^n$. We also fix $\vartheta_1, \ldots, \vartheta_m > 0$ and $w_1, \ldots, w_m >0 $  satisfying $$w_1  + \ldots + w_m = n.$$ 

Let $\X$ denote the space of affine unimodular lattices in $\R^{m+n}$. Consider the group $G:= SL_{m+n}(\R) \ltimes \R^{m+n}$,the group of all invertible affine transformations of $\R^{m+n}.$ The group structure on $G$ is given as $$(A,v).(B,w)=(AB, Aw+v) \text{   for all   } (A,v), (B,w) \in G.$$  We have a natural action of $G$ on $\X$, given by 
 \begin{equation}
\label{action2}
    ((A,v), \Lambda) \mapsto A.\Lambda +v  =\{A.w +v : w \in \Lambda\},
\end{equation}
for all affine unimodular lattice $\Lambda \in \X$.
 This gives us a surjection from $G $ onto $\X$, given by $(A,v) \mapsto A \Z^{m+n} + v$. Thus, we can identify $\X$ as a homogeneous space of the group $G$. More specifically, $\X \simeq G/\Gamma $, where $\Gamma := SL_{m+n}(\Z) \ltimes \Z^{m+n} $ is the kernel of the above surjective map. Let $\mu_{\X}$ denote the G-invariant probability measure on $\X$ and $m_G$ denote the bi-invariant Haar measure on G, normalized so that the fundamental region for the action of $\Gamma $ on $G$ has measure 1.  
 Denote by $U$ the subgroup
\begin{equation}
    U:= \left\{ \begin{pmatrix} I_m & u \\ 0 & I_{n+1} \end{pmatrix} : u \in M_{m \times (n+1)}(\R) \right\} < G.
\end{equation}
Let $\Y := U\Z^{m+n} = \{ u. \Z^{m+n}: u \in U\} \subset \X$. Geometrically, $\Y$ can be viewed a $m(n+1)$-dimensional torus embedded in the space of lattices $\X$. We denote by $\mu_{\Y}$ the probability measure on $\Y$ induced by Lebesgue measure on $M_{m \times (n+1)}([0,1])$. Note that elements of $\Y$ look like
 \begin{equation}
 \label{defLambda1}
     \Lambda_{u,v}= \left\{ \left( p_1+ \sum_{j=1}^n u_{1j}q_j + v_1,\ldots, p_m +\sum_{j=1}^{n} u_{mj}q_j + v_m, \bar{q}\right) : (\bar{p}, \bar{q}) \in \Z^m \times \Z^n \right\} 
 \end{equation}
 where $u \in M_{m \times n}([0,1])$, $v \in M_{m \times 1}([0,1])$.  \par

For sake of simplicity, we denote $d:= m+n$. Note that the group $G$ can be considered as subgroup of $G_{d+1}= \SL_{d+1}(\R)$, via the map $H: G \rightarrow G_{d+1}$ given by
\begin{equation}
\label{eq:aff def H}
    (A,u) \rightarrow \begin{pmatrix}
        A & u \\ 0 & 1
    \end{pmatrix} \text{    for } A \in SL_{m+n}(\R),\  u \in \R^{m+n}.
\end{equation}
Under this identification, it is clear that $\Gamma= G \cap \Gamma_{d+1}$, where $\Gamma_{d+1}= \SL_{d+1}(\Z)$. Thus $H$ induces an injective map $h: \X \rightarrow \X_{d+1}$, where $\X_{d+1} \simeq  G_{d+1}/ \Gamma_{d+1}$, denotes the space of all unimodular lattices in $\R^{d+1}.$ \\
Also, we have a surjective map  from $G$ onto $G_d= \SL_d(\R)$ given by $\Tilde{\pi}(A,v)=A$. It is clear that $\Tilde{\pi}(\Gamma)= \SL_d(\Z)$. Hence, $\Tilde{\pi}$ induces a surjective map $\pi: \X \rightarrow \X_d= \SL_d(\R)/ \SL_d(\Z). $ Note that the map $\pi: (\X, \mu_{\X}) \rightarrow (\X_d, \mu_{\X_d})$ is a measure preserving map, where $\mu_{\X_d}$ is the unique $G_d$ invariant probability measure on $\X_d$. To see this, note that the map $\pi$ is continuous, hence a measurable map. So we may consider the pushforward of the measure $\mu_{\X}$ under $\pi$, denoted by $\pi_*(\mu_{\X})$. Clearly, the $G$-invariance of $\mu_{\X}$  descends to give $G_d$-invariance of $\pi_*(\mu_{\X})$. Thus, $\pi_*(\mu_{\X})$ is a $G_d$-invariant probability measure on $\X_d$, hence must agree with $\mu_{\X_d}$ using the uniqueness of latter. 
We also define $$U_d = \left\{ \begin{pmatrix}
            I_m & u \\ 0 & I_n
\end{pmatrix}: u \in \M_{m \times n}(\R) \right\}. $$ Let $\Y_d$ be image of $U_d$ under the projection map $SL_{d}(\R) \rightarrow \X_d$. Let $\mu_{\Y_d}$ be the probability measure on $\Y_d$ induced by Lebesgue measure on $\M_{m \times n}([0,1])$. Then it is clear that $\pi : \X \rightarrow \X_d$ maps $\Y$ onto $\Y_d$. Moreover, the restriction map $\pi_{|_{\Y}} : (\Y, \mu_{\Y}) \rightarrow (\Y_d,\mu_{\Y_d})$ is a measure preserving map.

\subsection{Mixing of the $a^s\Y$ action on $\X$} \label{Aff Mix Property} \vspace{0.3in} \hfill \\

In order to state the mixing result for the $a^s\Y$ action on $\X$, we will need some notation. We will follow \cite{BG3} in this regard. Fix positive weights $w_1, w_2, \ldots, w_{m+n}$, satisfying $$\sum_{i=1}^{m} w_i = \sum_{i=m+1}^{m+n} w_i$$ and denote by $(a_t)$ the corresponding  one parameter semi-subgroup of $G$ given by $a_t=(a_t', 0)$, where $$a_t' ;= diag(e^{w_1t}, \ldots, e^{w_{m}t}, e^{-w_{m+1}t}, \ldots, e^{-w_{m+n}t}) \ \text{, } t>0. $$

We set $b_t=(b_t',0)$ to be the `equal weight' one parameter subgroup
\begin{equation}
\label{Hom b_t}
    b_t' := diag(e^{t/m}, \ldots , e^{t/m}, e^{-t/n}, \ldots, e^{-t/n}), \ t>0.
\end{equation}
This flow coincides with $(a_{t})$ with the special choice of exponents $$w_1= \ldots= w_m = \frac{1}{m}, \ \ \ w_{m+1}= \ldots= w_{m+n}= \frac{1}{n}.$$

Every $Y \in Lie(G)$ defines a first order differential operator $\mathcal{D}_Y$ on  $\C_c^\infty(\X)$ by \begin{equation*}
     \D_Y(\phi)(x) := \frac{d}{dt} \phi(exp(tY)x)|_{t=0}.
 \end{equation*}
Fixing an ordered basis ${Y_1,Y_2,\ldots, Y_r}$ of $Lie(G)$, every monomial $Z= Y_1^{l_1}\ldots Y_r^{l_r} $ defines a differential operator by
 \begin{equation}
 \label{diff hom}
     \D_Z := \D_{Y_1}^{l_1} \ldots  \D_{Y_r}^{l_r}
 \end{equation}
 of \textit{degree} $deg(Z)= l_1+ \ldots+l_r$. For $k \geq 1$ and $\phi \in \C_c^{\infty} (\X)$, define the norms 
 \begin{equation}
 \label{eq: def Lk}
    \|\phi\|_{L_k^2(\X)} := \left( \sum_{deg(Z) \leq k} \int_{\X} |(\D_Z \phi)(x)|^2 \, d\mu_{\X}(x) \right)^{1/2}
 \end{equation}
 and
 \begin{equation}
 \label{eq: def Ck}
     \|\phi\|_{C^k} := \sum_{deg(Z) \leq k} \|(\D_Z \phi)(x)\|_{\infty}.
 \end{equation}
 For every $g \in G$, $\phi \in \C^{\infty}(\X)$ and $Y \in Lie(G)$, we have $\D_Y(\phi \circ g) = \D_{Ad(g) Y}(\phi) \circ g$. This identity extends to the universal enveloping algebras $\U(Lie(G))$ as well, and thus we also have $\D_Z(\phi \circ g) = \D_{Ad(g)Z}(\phi) \circ g$, for every monomial $Z$ in ${Y_1, \ldots,Y_r}$, where $Ad(g)$ denotes the extension of the $Ad(g)$ from $Lie(G)$ to $\U(Lie(G))$. Since $Ad(g)Z $ can be written as a finite sum of monomials of degrees not exceeding the degree of $Z$, we conclude that for every $k \geq 1 $, there exists a sub-multiplicative function $g \mapsto C_k(g)$ such that
 \begin{equation*}
     \|\phi \circ g\|_{C^k} \leq C_k(g) \|\phi\|_{C^k}, \text{    for all } \phi \in \C_c^{\infty}(\X).
 \end{equation*}
 In particular, there is a constant $\xi = \xi(m,n,k)$ ( which also depends on fixed choice of weights $w_1, \ldots, w_{m+n}$) such that 
 \begin{equation}
 \label{xi ineq 2}
     \|\phi \circ a_t\| \ll e^{\xi t} \|\phi\|_{C^k}, \text{    for all $t \geq 0$ and } \phi \in \C_c^{\infty}(\X), 
 \end{equation}
 where the suppressed constants are independent of $t$ and $\phi$.

 The starting point of our discussion is to prove a quantitative estimate on correlations of smooth functions on $\X$. We first record two results from \cite{SE} and \cite{KM2}.

\begin{thm}[\cite{SE}, Theorem 10]
    \label{SamuelE}
    Let $$\Phi_t'= \begin{pmatrix}
    e^{\lambda_1t} & & \\
    & \ddots & \\
    & & e^{\lambda_dt}
  \end{pmatrix},$$
  where $\sum_{i=1}^d \lambda_i  =0.$ Given $\Gamma\backslash G \ni x= \Gamma. (M,v)$, we define $$\Phi_t(x) := \Gamma. (M,v). (\Phi_t', 0_{d}).$$ For $f,g \in \Cc(\Gamma\backslash G)$ such that $\int_{\T^d} f =0$ and $t \geq 0$, $$\int_{\Gamma \backslash G} f(x) g(\Phi_t(x))\, d\nu(x)= \Oo(e^{- \lambda_{\max}t} \|f\|_{C^d(\Gamma \backslash G)} \|g\|_{C^d(\Gamma \backslash G)}),$$ where $\lambda_{\max}= \max(\lambda_1, \ldots, \lambda_d),$ $\nu$ is unique $G$ invariant probability measure on $\Gamma \backslash G$.
    \end{thm}

\begin{thm}[\cite{KM2}, Cor 2.4.4]
 \label{mixing hom}
     There exists $\gamma > 0$ and $k \geq 1$ such that for all $\phi_1,\phi_2 \in \Cc(\X)$ and $g \in G$, 
     \begin{align}
         \int_{\X} \phi_1(gx) \phi_2(x) \, d\mu_{\X}(x) = \left( \int_{\X} \phi_1 \, d\mu_{\X}\right) \left( \int_{\X} \phi_2 \, d\mu_{\X} \right) \\ 
         + \mathcal{O}\left( e^{-\gamma dist(g,e)} \|\phi_1\|_{L_k^2(\X)}\|\phi_2\|_{L_k^2(\X)} \right)
     \end{align}
 \end{thm}

We are now ready for

\begin{thm} \label{Aff Mixing}
There exists $\gamma > 0$ and $k \geq 1$ such that for all $\phi_1, \phi_2 \in \C_c^{\infty}(\X)$ and $t \geq 0$,
\begin{equation}
    \int_{\X} \phi_1(b_tx) \phi_2(x) \, d\mu_{\X}(x) = \left( \int_{\X} \phi_1(x) \, d\mu_{\X}(x)\right) \left( \int_{\X} \phi_2(x) \, d\mu_{\X}(x) \right) + \mathcal{O}\left( e^{-\gamma t} \|\phi_1\|_{\C^k}\|\phi_2\|_{\C^k} \right).
\end{equation}
\end{thm}
\begin{proof}
Define a function $\psi_i : \X_d  \rightarrow \R$ as $\psi_i(A .\SL_d(\Z)) = \int_{\T^{d}} \phi_i((A, Aw).\Gamma) \, dw$ for $i=1,2$. Then we have
    \begin{align}
    \label{aff 1 1}
        &\int_{\X} \phi_1(b_tx) \phi_2(x) \, d\mu_{\X}(x) =\int_{\X}   \phi_1(b_tx) \left(\phi_2(x) - \psi_2(\pi(x)) \right)  \, d\mu_{\X}(x) + \int_{\X} \phi_1(b_t x)  \psi_2(\pi(x))  \, d\mu_{\X}(x)
    \end{align}
     Now, to estimate the first term in \eqref{aff 1 1}, we use Theorem \ref{SamuelE}. Notice that we have a measure space isomorphism $\Psi: (\Gamma \backslash G, \nu) \rightarrow (\X, \mu_{\X})$ given by $$\Psi(\Gamma. (A,v))= (A,v)^{-1}.\Gamma= (A^{-1}, -A^{-1}v). \Gamma.$$ Using this isomorphism, define $f,g \in \Cc( \Gamma \backslash G) $ as 
\begin{align*}
    f(\Gamma. (A,v)) &= (\phi_2 - \psi_2 \circ \pi ) \circ \Psi (\Gamma. (A,v)) = \phi_2((A^{-1}, -A^{-1}v). \Gamma) -  \psi_2(A^{-1} \SL_d(\Z)), \\
    g(\Gamma. (A,v)) &= \psi_1 \circ \Psi (\Gamma. (A,v))= \psi_1((A^{-1}, -A^{-1}v). \Gamma)
\end{align*}
It is clear that we have $\int_{\T^d} f =0$. Now, apply Theorem \ref{SamuelE} with $\Phi_t'= (b_t')^{-1}$ to get that
\begin{align*}
    \int_{\X}   \phi_1(b_tx) \left(\phi_2(x) - \psi_2(\pi(x)) \right)  \, d\mu_{\X}(x) &= \int_{\Gamma \backslash G} \phi_1(b_t\Psi(x)) \left(\phi_2(\Psi(x)) - \psi_2(\pi(\Psi(x))) \right)\, d\nu(x) \\
    &= \int_{\Gamma \backslash G} f(x) g(\Phi_t(x))\, d\nu(x)\\
    & \ll e^{-t/n} \|f\|_{C^d(\Gamma \backslash G)} \|g\|_{C^d(\Gamma \backslash G)} \\
    & \ll e^{-t/n} \|\phi_1\|_{\C^d}\|\phi_2\|_{\C^d}
\end{align*}

To estimate second term in \eqref{aff 1 1}, we note that the fundamental domain of $G / \Gamma$ equals the disjoint union of $A \times \{Av : v \in [0,1)^{m+n}\}$ where $A$ belongs to the fundamental region of the action of $\SL_d(\Z)$ on $\SL_d(\R)$, say $\mathcal{F}$. If we denote by $m_{G_d}$, the unique Haar measure on $G_d$ such that $m_{G_d}(\mathcal{F})=1$, then we have     
\begin{align}
        \int_{\X} \phi_1(b_t x)  \psi_2(\pi(x))  \, d\nu(x)  &= \int_{\mathcal{F}} \int_{\T^{m+n}} \phi_1( b_t((A,Av).\Gamma)) \psi(A. \SL_d(\Z)) \, dv dm_{G_d}(A) \nonumber  \\
        &= \int_{\mathcal{F}} \psi_1(b_t' A.\SL_d(\Z)) \psi_2(A. \SL_d(\Z)) \,  dm_{G_d}(A) \nonumber \\
        &= \int_{\X_d} \psi_1(b_t'y) \psi_2(y) \,  d\mu_{\X_d}(y)  \nonumber \\
        &= \left( \int_{\X_d} \psi_1 \, d\mu_{\X_d}\right) \left( \int_{\X_d} \psi_2 \, d\mu_{\X_d}\right) + \Oo(e^{- \gamma t} \|\psi_1\|_{C^k(\X_d)} \|\psi_2\|_{C^k(\X_d)})  \label{aff 1 2}\\
        &= \left( \int_{\X} \phi_1(x) \, d\mu_{\X}(x)\right) \left( \int_{\X} \phi_2(x) \, d\mu_{\X}(x) \right) + \mathcal{O}\left( e^{-\gamma t} \|\phi_1\|_{\C^k}\|\phi_2\|_{\C^k} \right) 
\end{align}
where the equality \eqref{aff 1 2} holds by Theorem \ref{mixing hom}. 

This completes the proof.
\end{proof}

\emph{From now  on  we  fix  $k \geq 1$   so   that  Theorem  \ref{Aff Mixing}\  holds}.

We fix a right invariant metric $dist$ on $\X_{d+1} \simeq G_{d+1}/ \Gamma_{d+1}= \SL_{d+1}(\R)/ \SL_{d+1}(\Z)$ induced from a right-invariant Riemannian metric on $G_{d+1}$. We denote by $B_{{G_{d+1}}}(\rho)$ the ball of radius $\rho$ centered at identity in ${G_{d+1}}$. For a point $x \in {\X_{d+1}}$, we let ${i_{d+1}}(x)$ denote the $injectivity \ radius$ at $x \in \X_{d+1}$; namely the supremum over $\rho > 0$ such that the map $B_{{G_{d+1}}}(\rho) \rightarrow B_{{G}}(\rho)x : g \mapsto gx$ is injective.

Since we can identify $\X$ as subset of $\X_{d+1}$ via the map $h$, defined as in Section \ref{Aff Intro}. So, we get a right invariant metric on $\X$ given by $dist_{|\X \times \X}$ . In exact same way, we define the injectivity radius at $x \in \X$ , denoted by $i(x)$, as the supremum over $\rho > 0$ such that the map $B_{{G}}(\rho) \rightarrow B_{{G}}(\rho)x : g \mapsto gx$ is injective, where $B_{{G}}(\rho)$ the ball of radius $\rho$ centered at identity in ${G}$. Then it is clear that 
\begin{equation}
    \label{eq:aff inj r}
    i(x) \geq i_{d+1}(h(x)).
\end{equation}

Now, define for $\epsilon >0$,  $\K_{\e}:= \pi^{-1}(\K_{\e,d})$, where again $\K_{\e,d}$ is defined as 
\begin{equation*}
    {\K}_{\epsilon,d} =\{ \Lambda \in {\X_{d}}: \|v\| > \epsilon, \text{   for all } v \in \Lambda \setminus \{0\} \}.
\end{equation*} 
Since, the map $\pi$ is a proper map and $\K_{\e,d}$ is compact by Mahler's criterion, we get that $\K_{\e}$ is a compact subset of $\X$. % As a topological space, we have G/ \Gamma \simeq \Gamma \backslash G \simeq \SL_{d}(\Z) \backslash \SL_{d}(\R) \times [0,1]^d \simeq \SL_{d}(\R)/ \SL_d(\Z) \times [0,1]^d. With this identification, the map $\pi$ is given by projection onto first co-ordinate. Clearly, it is closed continuos surjection with compact fibres, hence proper.

\begin{prop}
    \label{inj r aff}
   $i(x) \gg \epsilon^{d+1}$ for all $x \in \K_{\epsilon}$ and $0 < \e <1$.
\end{prop}
\begin{proof}
    Note that for $\e <1$, we have the equality $\K_{\e} = h^{-1}(\K_{\e,d+1}) \subset \X$ where 
\begin{equation*}
    {\K}_{\epsilon,d+1} =\{ \Lambda \in {\X_{d+1}}: \|v\| > \epsilon, \text{   for all } v \in \Lambda \setminus \{0\} \}.
\end{equation*} 
Hence, the proof follows from equation \eqref{eq:aff inj r} and the following proposition
\begin{prop} [\cite{KM1} Prop 3.5] 
\label{inj r hom}
$i_{d+1}(x) \gg \epsilon^{d+1}$ for all $x \in \K_{\epsilon,d+1}$.
\end{prop}
\end{proof}

Our next task is to prove the followig effective equidistribution theorem.
\begin{thm}
\label{aff main lemma for mixing}
    There exists $\rho_0 > 0$ and $c,\gamma >0 $ such that for every $\rho \in (0,\rho_0),$ $f \in \C_c^{\infty}(U)$ satisfying $supp(f) \subset B_G(\rho)$, $x \in \X$ with $i(x) > 2\rho$, $\phi \in \C_c^{\infty}(\X)$, and $t \geq 0$, 

    \begin{equation}
        \int_U f(u) \phi(b_t u x ) \, du = \left( \int_U f(u) \,du \right) \left( \int_{\X} \phi \, d\mu_{\X} \right) + \mathcal{O} \left( \rho \|f\|_{L^1(U)} \|\phi \|_{Lip}  + \rho^{-c} e^{-\gamma t} \|f\|_{C^k} \|\phi\|_{\C^k} \right)
    \end{equation}
\end{thm}

\begin{proof}

The proof is motivated from proof of [Theorem 2.3 of \cite{KM1}]. In this proof we identify $G$ as the subgroup of $G_{d+1}= SL_{m+n+1}(\R)$ via the map $H$, defined as in Section \ref{Aff Intro}.
Let us define 
\begin{equation*}
    U^- = \left\{ \begin{pmatrix}
        I_m & 0 & 0 \\  \alpha & I_n & \beta \\ 0 & 0 & 1
    \end{pmatrix}  : \alpha \in \M_{n \times m}(\R), \beta \in \M_{n \times 1}(\R) \right\}
\end{equation*}
and 
\begin{equation*}
    U^0 = \left\{ \begin{pmatrix}
        A & 0 & 0 \\  0 & B & 0 \\ 0 & 0 & 1
    \end{pmatrix}  : A \in \GL_{m}(\R), B \in \GL_{n}(\R), det(A)det(B)=1 \right\}.
\end{equation*}
Then the product map $U^- \times U^0 \times U \rightarrow G$ is a local diffeormorphism. Define $\rho_0$ so that the inverse map is well defined on $B_G(2\rho_0)$. Note that $U^-$ is expanding horospherical with respect to $ b_{-t}, t>0$, while $U^0$ is centralized by $\{b_t\}$. Thus, the inner automorphism of G given by $h \rightarrow b_t h(b_t)^{-1}$ is non-expanding on the group $\Tilde{U} := U^-U^0$ and the latter is 
\begin{equation*}
 \left\{ \begin{pmatrix}
        A & 0 & 0 \\  \alpha & B & \beta \\ 0 & 0 & 1
    \end{pmatrix}  : A \in \GL_{m}(\R), B \in \GL_{n}(\R), det(A)det(B)=1 ,\alpha \in \M_{n \times m}(\R), \beta \in \M_{n \times 1}(\R) \right\}.
\end{equation*}
In fact,  one has 
\begin{equation} \label{A_0}
     \forall~r>0 , \ \ \forall~t>0, \ \ b_t(B_{\Tilde{U}}(r))b_t^{-1} \subset B_{\Tilde{U}}(r).
\end{equation}
Let $\nu$ denote the Lebesgue measure on $U \cong \R^{m \times (n+1)}$. Let us choose Haar measures $\nu^-$, $\nu^0$ on $U^-$, $U^0$ respectively, normalized so that $m_G$ is locally almost the product of  $\nu^-$, $\nu^0$ and $\nu$. Then, \textbf{ by [\cite{EW}, Lem. 11.31]}, $m_G$ can be expressed via $\nu^-$, $\nu^0$ and $\nu$ in the following way: for any $\varphi \in L^1(G)$
\begin{equation}
    \label{A_1}
    \int_{U^- U^0 U} \varphi(g) \, dm_G(g)  = \int_{U^- \times U^0 \times U} \varphi(u^- u^0 u) \Delta(u^0) \, d\nu^-(u^-) d\nu^0(u^0) d\nu(u)
\end{equation}
where $\Delta$ is the modular function of (the non-unimodular group) $\Tilde{U}$.

 Now, we are given $\rho \in (0, \rho_0)$, $f \in \C_c^{\infty}(U)$ satisfying $supp(f) \subset B_G(\rho)$, $x \in \X$ with $i(x) > 2\rho$, $\phi \in \C_c^{\infty}(\X)$, and $t>0$. Without loss of generality, we may assume that $\int_{\X}\phi = 0$, otherwise subtract $\int_{\X}\phi$ from $\phi$. We will also need the following Lemma.
 \begin{lem} [\cite{KM2}, Lemma 2.4.7] 
\label{lem_KM}

(a) For any $r>0$, there exists a non-negative function $\theta \in \C_c^{\infty}(\R^d)$ such that $supp(\theta)$ is inside $B(r)$, $\int_{\R^d}\theta =1$, and $\|\theta\|_{\C^k} \ll r^{-(k + N)}$.

(b) Given $\theta_1, \theta_2 \in \C_c^{\infty}(\R^d)$, then the function $\theta(x)= \theta_1(x) \theta_2(x)$ 
satisfy $\theta \in \C_c^{\infty}(\R)$ and $\|\theta\|_{\C^k} \ll \|\theta_1\|_{\C^k}\|\theta_2\|_{\C^k}$.

(c) Given $\theta_1 \in \C_c^{\infty}(\R^{d_1})$, $\theta_2 \in \C_c^{\infty}(\R^{d_2})$, define the function 
 $\theta(x) \in \C_c^{\infty}(\R^{d_1 + d_2})$ by  $\theta(x_1,x_2)= \theta_1(x_1).\\ \theta_2(x_2)$. Then $\|\theta\|_{\C^k} \ll \|\theta_1\|_{\C^k} \|\theta_2\|_{\C^k}$.
\end{lem}

\begin{rem}
In \cite{KM2} the Lemma was proved for Sobolev norms $\|.\|_l$ instead of $\|.\|_{\C^k}$, but the same statement holds for $\|.\|_{\C^k}$.
\end{rem}
 
 Using, Lemma \ref{lem_KM}, one can choose non-negative functions $\theta^- \in \C_c^{\infty}(U^-)$,  $\theta^0 \in \C_c^{\infty}(U^0)$ with 
 \begin{equation}
     \label{A_2}
     \int_{U^-} \theta^- = \int_{U^0} \theta^0 = 1
 \end{equation}
such that
\begin{equation}
    \label{A_3}
    supp(\theta^-).supp(\theta^0) \subset B_{\Tilde{U}}(\rho),
\end{equation}
and at the same time 
\begin{equation}
    \label{A_4}
    \|\Tilde{\theta}\|_{\C^k} \ll \rho ^{-(2k + 2N)}, \text{   with } N= n^2 + m^2 + nm + n -1
\end{equation}
where $\Tilde{\theta} \in \C_c^{\infty} (\Tilde{U})$ is defined by
\begin{equation}
    \label{A_5}
    \Tilde{\theta}(u^- u^0) := \theta^-(u^-) \theta^0 (u^0) \Delta(u^0)^{-1}.
\end{equation}
Also define $\varphi \in \C_c^{\infty} (\X)$ by $\varphi(u^-u^0 ux)= \Tilde{\theta}(u^- u^0)f(h)$; the definition makes sense because of \eqref{A_3} and the assumptions on $f$ and $x$.
Then we have,
\begin{align*}
   & \left| \int_{U} f(u)\phi(b_t u x )\, du - \int_{\X} \varphi(y) \phi(b_t y) \,d\mu_{\X}(y) \right| \\
&\underset{supp(\varphi) \subset B_G(2\rho)x}{=}  \left| \int_{U} f(u)\phi(b_t u x)\, du -  \int_{G} \varphi(gx) \phi(b_t g x) \,dm_G(g)\right| \\
    & \underset{\eqref{A_1}}{=}  \left| \int_{U} f(u)\phi(b_t u x)\, du -  \int_{U^- \times U^0 \times U} \Tilde{\theta}(u^-u^0) f(u) \phi(b_t u^- u^0 u x)\Delta(u^0)^{-1} \,d\nu^-(u^-) d\nu^0(u^0) d\nu(u)\right| \\
    &\underset{\eqref{A_2}, \eqref{A_5}}{=} \left|\int_{U^- \times U^0 \times U} \theta^-(u^-) \theta^0(u^0) f(u) (\phi(b_t u x) -\phi(b_t u^- u^0 ux)) \,d\nu^-(u^-) d\nu^0(u^0) d\nu(u)\right| \\
    & \underset{\eqref{A_2}, \eqref{A_5}}{=} \left|\int_{U^- \times U^0 \times U} \theta^-(u^-) \theta^0(u^0) f(u) (\phi(b_t u x) -\phi((b_t u^- u^0 b_t^{-1}) b_t ux)) \,d\nu^-(u^-) d\nu^0(u^0) d\nu(u)\right| \\
   & \underset{\eqref{A_0}, \eqref{A_3}}{\leq } \underset{g \in B_{\Tilde{U}}(\rho), y \in \X }{sup} |\phi(gy)- \phi(y)| .\int_{U^- \times U^0 \times U} |\theta^-(u^-) \theta^0(u^0) f(u)| \,d\nu^-(u^-) d\nu^0(u^0) d\nu(u) \\
   & \leq  \|\phi\|_{Lip}. r. \int_U|f|.
\end{align*}
    
Now, in view of Theorem \ref{Aff Mixing} we have
\begin{equation}
  \left| \int_{\X} \varphi(y) \phi(b_t y) \,d\mu_{\X}(y) \right| \ll   \|\varphi\|_{\C^k} .\|\phi\|_{\C^k} .e^{-\gamma t}. 
\end{equation}
On  the other hand, in view of Lemma \ref{lem_KM} and fact the projection map from $G \rightarrow \X$ is a local isometry, we get
\begin{equation*}
    \|\varphi\|_{\C^k} = \|\Tilde{\theta}. f\| \ll \|\Tilde{\theta}\|_{\C^k}\|f\|_{\C^k} \underset{\eqref{A_4}}{\ll} \rho^{-(2k + 2N)} \|f\|_{\C^k}.
\end{equation*}
Hence the proof.

\end{proof}

We will embed the flow $(a_t)$ in a multi-parameter flow as follows. 
Denote by $S^+$ the cone in $\R^{m+n}$ consisting of those $\bar{s}= (s_1,\ldots, s_{m+n})$ which satisfy
\begin{equation*}
    s_1,\ldots, s_{m+n} >0 \text{   and    } \sum_{i=1}^m s_i = \sum_{i= m+1}^{m+n} s_i.
\end{equation*}
For $\bar{s}= (s_1, \ldots, s_{m+n}) \in \R^{m+n}$, we set 
\begin{equation*}
    a(\bar{s}) := (a'(\bar{s}),0_d) \in G,
\end{equation*}
where $$a'(\bar{s})= diag(e^{s_1}, \ldots, e^{s_{m}}, e^{-s_{m+1}}, \ldots, e^{-s_{m+n}}).$$
For $\bar{s}=(s_1, \ldots, s_{m+n}) \in S^+ $, we also define  $$\lfloor \bar{s}\rfloor := \min(s_1,\ldots, s_{m+n}).$$
With $\bar{s_t}:= (w_1t,\ldots, w_{m+n}t)$, we see that $a_t = a(\bar{s}_t)$.

\begin{thm}
\label{aff lemma 2}
    There exists $\theta= \theta(m,n) >0$ such that for every compact $L \subset \X$ and a Euclidean ball $ B \subset U$ centered at the identity, there exists $T_0 >0$ such that for every $\epsilon \in (0,1)$, $x \in L$, and $\bar{s} \in S^+$ satisfying $\lfloor \bar{s}\rfloor \geq T_0$, one has 
\begin{equation*}
    |\{ u \in B: a(\bar{s})ux \notin \K_{\epsilon}\}| \ll \epsilon^{\theta} |B|
\end{equation*}
\end{thm}
\begin{proof}
Let us define $B'= \pi(B) \subset U_d$ and $B''$ denote image of $B$ under the projection map from $U = U_d \ltimes (\R^m \times \{0\}^n) \rightarrow \R^m$. If $r$ is radius of ball $B$, then upto a constant, we have $|B| \approx r^{m(n+1)}$, $|B'| \approx r^{mn}$ and $|B''| \approx r^m$.Now, for every $x \in \X$ we have
\begin{align*}
    |\{ u \in B: a(\bar{s})ux \notin \K_{\epsilon}\}| 
    &= |\{ u \in B: a(\bar{s})ux \notin \pi^{-1}(\K_{\epsilon,d})\}| \\
    &= |\{ u \in B: \pi(a(\bar{s})ux) \notin \K_{\epsilon,d}\}| \\
    &= |\{ (u',v') \in U_d \times (R^m \times \{0\}^{n}):  (u',v') \in B, \ \ a'(\bar{s}).u'.\pi(x) \notin {\K_{\epsilon,d}}  \}|\\
    &\leq |\{u' \in B':  a'(\bar{s}).u'.\pi(x) \notin {\K_{\epsilon,d}}\}| |\{ v' \in \R^m \times \{0\}^n: v' \in B'' \times \{ 0\}^n  \}| \\ 
   &\leq |\{u' \in B':  a'(\bar{s}).u'.\pi(x) \notin {\K_{\epsilon,d}}\}| |B''| \\
    &\ll |\{u' \in B':  a'(\bar{s}).u'.\pi(x) \notin {\K_{\epsilon,d}}\}|. r^{m} \\
    &\ll \frac{ |\{u' \in B':  a'(\bar{s}).u'.\pi(x) \notin {\K_{\epsilon,d}}\}|}{ |B'|} |B| 
\end{align*}
for all $x \in \X$. Now using the fact that $L'= \pi(L)$ is a compact subset of $\X_d$, the result follows from the following Theorem. 
\begin{thm} [\cite{KM1}, Cor 3.4]
\label{hom lemma 2}
    There exists $\theta= \theta(m,n) >0$ such that for every compact $L' \subset \X$ and a Euclidean ball $ B' \subset U$ centered at the identity, there exists $T_0 >0$ such that for every $\epsilon \in (0,1)$, $x \in L$, and $\bar{s} \in S^+$ satisfying $\lfloor \bar{s}\rfloor \geq T_0$, one has 
\begin{equation*}
    |\{u' \in B':  a'(\bar{s}).u'.\pi(x) \notin {\K_{\epsilon,d}}\}| \ll \epsilon^{\theta} |B'|
\end{equation*}
\end{thm}
   
\end{proof}

With all the technical work done above, the following Theorems follow from the proofs of corresponding theorems in \cite{BG} along the same lines.  
\begin{thm}
    There exists  $\delta > 0$ such that for every compact $\Omega \subset U$, $f \in \C_c^{\infty}(U)$ with $supp(f) \subset \Omega$, $\phi_1, \ldots, \phi_r \in \C_c^{\infty}(\X)$, $x_0 \in \X$, and $t_1, \ldots,t_r >0 $, we have 
    \begin{equation*}
        \int_U f(u) \left( \prod_{i=1}^r \phi_i(a_{t_i} u x_0)\right) \, du = \left(\int_U f(u) \, du  \right) \prod_{i=1}^r \left( \int_{\X} \phi_i \,d\mu_{\X} \right) + \mathcal{O}_{x_0, \Omega, r} \left(e^{-\delta D(t_1, \ldots, t_r)} \|f\|_{C^k}  \prod_{i=1}^r \| \phi_i\|_{\C^k} \right)
    \end{equation*}
 where
 \begin{equation}
     D(t_1, \ldots, t_r) := \min\{t_i, |t_i - t_j|: 1 \leq i \neq j \leq r\}.
 \end{equation}
\end{thm}

\begin{cor}
    There exists  $\delta' > 0$ such that for $\phi_0 \in \C_c^{\infty}(\Y )$, $\phi_1, \ldots, \phi_r \in \C_c^{\infty}(\X)$ and $t_1, \ldots,t_r >0 $, we have 
    \begin{align*}
        \int_{\Y} \phi_0(y) \left( \prod_{i=1}^r \phi(a_{t_i} y)\right) \, d\mu_{\Y} &= \left(\int_{\Y} \phi_0 \, d\mu_{\Y}  \right) \prod_{i=1}^r \left( \int_{\X} \phi(a_{t_i} u x_0) \,d\mu_{\X} \right) \\  &+ \mathcal{O}_{x_0, \Omega, r} \left(e^{-\delta D(t_1, \ldots, t_r)} \|f\|_{C^k}  \prod_{i=1}^r \|\phi_i\|_{\C_k} \right)
    \end{align*}
\end{cor}

\begin{thm}
    Let $a= (diag(a_1, \ldots, a_{m+n}), 0)$ where 
    $$a_1,\ldots, a_m >1,\ \  0< a_{m+1}, \ldots , a_{m+n}<1 ,\ \  \text{and} \ \  a_1 \ldots a_{m+n}=1.$$
    Then $a$ defines a continuous self-map of the space $\X$, which preserves $\mu_{\X}$. Then for $\phi \in \C_c^{\infty}(\X)$ and $\xi \in \R$, we have
    \begin{equation*}
        \mu_{\Y}(\{y \in \Y: \frac{1}{\sqrt{N}} \sum_{s=0}^{N-1}(\phi \circ a^s - \mu_{\Y}(\phi \circ a^s))(y) < \xi\}) \rightarrow Norm_{\sigma_{\phi}(\xi)}
    \end{equation*}
    as $N \rightarrow \infty$, where
    \begin{equation*}
    \sigma_{\phi}^2 := \sum_{s= -\infty}^{\infty} \left( \int_{\X} (\phi \circ a^s) \phi \, d\mu_{\X} - \mu_{\X}(\phi)^2\right)
    \end{equation*}
\end{thm}

\subsection{Siegel Transforms} \label{Aff Siegel} \vspace{0.3in} \hfill \\

This section is devoted to a discussion of Siegel transforms. 
\subsubsection{Properties of Siegel Transforms} \hfill \\
  Given $f: \R^{m+n} \rightarrow \R$, we define its \textit{Siegel transform} $\hat{f}: \X \rightarrow \R$ by $$ \hat{f}(\Lambda):= \sum_{z \in \Lambda \setminus \{0\}} f(z), \text{   for } \Lambda \in \X.$$

Let $s \in \N$. Consider $\X_s$, the space of all unimodular lattices in $\R^s$, equipped with unique $G_s= SL_s(\R)$ invariant probability measure $\mu_{\X_s}$. Given a lattice $\Lambda$ in $\R^s$, we say that a subspace $V$ of $\R^s$ is $\Lambda$-rational if the intersection $V \cap \Lambda$ is a lattice in $V$. If $V$ is $\Lambda$-rational, we denote by $d_{\Lambda}(V)$ the volume of $V/ (V \cap \Lambda)$, and define
\begin{equation}
\label{def alpha_d}
    \alpha_s(\Lambda) =sup\{d_{\Lambda}(V)^{-1} : \text{  V is $\Lambda$-rational subspace of $\R^s$}\}.
\end{equation}
It follows from Mahler's Compactness Criterion that $\alpha_s$ is a proper function on $\X_s.$

Now, define the function $\alpha : \X  \rightarrow \R$ as 
 \begin{align}
  \alpha(x) &= \alpha_{d+1}(h(x)) = \alpha_{d}(\pi(x)),
\end{align}
where $\alpha_s$ is defined in \eqref{def alpha_d}. Note that the definition makes sense because of the fact that 
\begin{equation}
    \alpha_{d+1}\left( \begin{pmatrix}
        A & v \\ 0 & 1
    \end{pmatrix} \Z^{d+1} \right) = \alpha_d(A\Z^d) \text{,  } A \in SL_{d}(\R), v \in \R^d.
\end{equation}
%\textbf{Maybe insert a short proof} 
This is easy to observe. Assume that $\Lambda = \begin{pmatrix}
        A & v \\ 0 & 1
    \end{pmatrix} \Z^{d+1} $ and $\Lambda'= A\Z^d$. Suppose $V \subset \R^d$ is $\Lambda'$-rational, then $V \times \{0\}$ is $\Lambda$-rational. It is clear that $d_{\Lambda'}(V)= d_{\Lambda}(V \times \{0\})$, which gives that $\alpha_{d+1}(\Lambda) \geq \alpha_d(\Lambda).$  Now, if $W$ is $\Lambda$-rational, then we have two case. In first case, $W \subset \R^d \times \{0\}$, for which we immediately have $d_{\Lambda}(W)^{-1} \leq \alpha_d(\Lambda). $ Otherwise, we can write $W$ as $W= (W' \times \{0\}) \oplus \R(v,r)$ where $W'$ is $\Lambda'$-rational subspace and $(v,r) \in \Lambda$ with $r \in \Z \setminus \{0\}. $ It is again clear that $d_\Lambda(W)= d_{\Lambda'}(W'). |r|$. Since $|r| \geq 1$, we get that $d_{\Lambda}(W)^{-1} \leq d_{\Lambda'}(W')^{-1}\leq \alpha_d(\Lambda). $ This gives that $\alpha_{d+1}(\Lambda) \leq \alpha_d(\Lambda)$, and hence $\alpha_{d+1}(\Lambda) = \alpha_d(\Lambda).$
Now we have the following Proposition.

\begin{prop} 
\label{aff alpha 1 }
    If $f: \R^{m+n} \rightarrow \R$ is a bounded function with compact support, then 
    \begin{equation*}
        |\hat{f}| \ll_{supp(f)} \|f\|_{C^0} \alpha(\Lambda) \text{    for all } \Lambda \in \X.
    \end{equation*}
\end{prop}
\begin{proof}
    Without loss of generality, we may assume that $f$ is a non-negative function, otherwise replace $f$ by $|f|$. Let us define $\Tilde{f} : \R^{m+n+1} \rightarrow \R$ as $$\Tilde{f}(x,y) = \begin{cases} f(x), \text{y=1} \\ 0, \text{otherwise}\end{cases}. $$ Consider the usual Siegel transform $\hat{\Tilde{f}} : {\X}_{d+1} \rightarrow \R$, i.e., $\hat{\Tilde{f}}(\Lambda) = \sum_{v \in \Lambda \setminus \{0\}} \Tilde{f}(v)$. Note that if $\Lambda = \Lambda' +v = A\Z^d + v \in \X$ is an affine lattice in $\R^d$, then the lattice $h(\Lambda) \in \X_{d+1}$ is given by 
 \begin{align*}
 h(\Lambda) &= h((A,v) \Z^d) \\
            &= \begin{pmatrix}
                A & v \\ 0 & 1  
            \end{pmatrix} \Z^{d+1} \\ 
            &= \{(x+ rv,r) \in \R^{d+1}: x \in A\Z^d  , r\in \Z \} \\
            &= \{(x + rv,r) \in \R^{d+1}: x \in \Lambda' , r\in \Z \}. 
 \end{align*}
From this description, it is now clear that $|\hat{\Tilde{f}}(h(x))| = |\hat{f}(x)| + |f(0)| \geq |\hat{f}(x)| $ (since $f$ is a non-negative function). Now, we use the following Proposition.
\begin{prop}[\cite{SW}, Lem. 2]
\label{alpha 1}
If $\phi: \R^{d+1} \rightarrow \R$ is a bounded function with compact support, then $$|\hat{\phi}(\Lambda)| \ll_{supp(\phi)} \|\phi\|_{C^0} \alpha_{d+1}(\Lambda) \text{  for all } \Lambda \in \X_{d+1},$$
where $\hat \phi$ denote usual Siegel transform of $\phi$ on $\X_{d+1}$, the space of all uni-modular lattices in $\R^{d+1}$.
\end{prop}

From Lemma \ref{alpha 1}, we have $|\hat{\Tilde{f}}(h(\Lambda))| \underset{supp(\Tilde{f})}{\ll} \|\Tilde{f}\|_{C^0} \alpha_{d+1}(h(\Lambda)) = \|{f}\|_{C^0} \alpha(\Lambda) $ for all $\Lambda \in \X$. Hence, the result follows.
\end{proof}

Let us recall the following well known result in the Geometry of Numbers (see e.g. [\cite{EM}, Lemma 3.10]).
\begin{prop}\label{usedinaffalpha2}
    \label{alpha 2}
    $\alpha_d \in L^p(\X_d)$ for $1 \leq p < d$. In particular, $$\mu_{\X_d}(\{ \alpha_d \geq L\}) \ll_p L^{-p} \text{  for all }  p <d.$$ 
\end{prop}

We will use it to prove

\begin{prop}
\label{aff alpha 2}
    $\alpha \in L^p(\X)$ for $1 \leq p < m+n$. In particular, $\mu_{\X}(\{\alpha \geq L\}) \ll_p L^{-p}$ for all $p < m+n$. 
\end{prop}
\begin{proof}
    Since, the map $\pi: \X \rightarrow \X_d$ is measure preserving and $\alpha = \alpha_{d} \circ \pi$, we get that $\alpha \in L^p(\X)$ if and only if $\alpha_{d} \in L^p(\X_d)$. But, the latter follows from Proposition \ref{usedinaffalpha2}. The second part now follows from first one. 
\end{proof}

We now state an analogue of Rogers's higher moment formula for affine lattices.
\begin{prop}
\label{aff S R}
 If $f\in L^1(\R^d) \cap L^2(\R^d)$ , then
\begin{align}
\label{Af Se1}
    \int_{\X} \hat{f}(\Lambda) \, d\mu_{\X}(\Lambda) &= \int_{\R^{d}}f \\
\label{Af Rog1}
    \int_{\X} (\hat{f}(\Lambda))^2 \, d\mu_{\X}(\Lambda) &= \left(\int_{\R^{d}}f \right)^2 + \int_{\R^{m+n}} f^2
\end{align}

\begin{comment}
(b) If $F: \R^{m+n} \times \R^{m+n} \rightarrow \R$ is bounded integrable function with compact support, then 
\begin{equation}
 \label{Af Rog2}
    \int_{\X} \left( \sum_{v_1, v_2} F(v_1, v_2)\right) \, d\mu_{\X}(\Lambda)  = \int_{\R^{m+n} \times \R^{m+n}}F(v_1, v_2) \, d(v_1,v_2) + \int_{\R^{m+n}}F(v,v) \, dv
\end{equation}
\end{comment}
\end{prop}
\begin{proof}
The proof is well known and can be found in \cite{JA} or \cite{GH}. The latter contains a more detailed proof  and proves a more generalized version of above Proposition. The proof of \eqref{Af Rog1} can also be found in \cite{BMV}. We provide a proof for completeness. 

We may assume that $f$ is a non-negative function. Let us denote by $\mathcal{F}$, a fundamental region of action of $\SL_d(\Z)$ on $\SL_d(\R)$. Then we have, 
\begin{align*}
   \int_{\X} \hat{f}(\Lambda) \, d\mu_{\X}(\Lambda) &= \int_{\mathcal{F}} \int_{\T^{d}} \hat{f}( (A,Av).\Gamma) \, dv \, dm_{G_d}(A) \\
   &= \int_{\mathcal{F}}  \int_{[0,1]^d} \sum_{\substack{ \bar{p} \in \Z^d \\ A \bar{p} +Av \neq 0 } } {f}(A \bar{p} +Av) \, dv \, dm_{G_d}(A) \\
   &= \int_{\mathcal{F}}  \int_{[0,1]^d} \sum_{ \bar{p} \in \Z^d } {f}(A \bar{p} + Av) \, dv \, dm_{G_d}(A) \\
   &= \int_{\mathcal{F}}  \int_{\R^d}  {f}( Av) \, dv \, dm_{G_d}(A) \\
   &= \int_{\mathcal{F}}  \int_{\R^d}  {f}(v) \, dv \, dm_{G_d}(A) \\
   &= \int_{\R^d}  {f}( v) \, dv,
\end{align*}
which proves \eqref{Af Se1}. Also, we have 
\begin{align*}
    &\int_{\X} (\hat{f}(\Lambda))^2 \, d\mu_{\X}(\Lambda) = \int_{\mathcal{F}} \int_{\T^{d}} (\hat{f}( (A,Av).\Gamma))^2 \, dv \, dm_{G_d}(A) \\
    &= \int_{\mathcal{F}}  \int_{[0,1]^d} \sum_{\substack{ \bar{p} \in \Z^d \\ A \bar{p} +Av \neq 0 } } \sum_{\substack{ \bar{q} \in \Z^d \\ A \bar{p} +Av \neq 0 } } {f}(A \bar{p} +Av) {f}(A \bar{q} +Av) \, dv \, dm_{G_d}(A) \\
   &= \int_{\mathcal{F}}  \int_{[0,1]^d} \sum_{ \bar{p} \in \Z^d } \sum_{ \bar{q} \in \Z^d } {f}(A \bar{p} + Av) {f}(A \bar{q} + Av) \, dv \, dm_{G_d}(A) \\
   &= \int_{\mathcal{F}}  \int_{[0,1]^d} \sum_{ \bar{p} \in \Z^d } ({f}(A \bar{p} + Av))^2  \, dv \, dm_{G_d}(A) +  \int_{\mathcal{F}}  \int_{[0,1]^d} \sum_{ \bar{p} } \sum_{ \bar{q} \neq \bar{p} } {f}(A \bar{p} + Av) {f}(A \bar{p} + Av) \, dv \, dm_{G_d}(A) \\
    &{=} \int_{\mathcal{F}}  \int_{\R^d}  {f}(Av) \, dv \, dm_{G_d}(A) +  \int_{\mathcal{F}}  \int_{[0,1]^d} \sum_{ \bar{p} } \sum_{ \bar{q} \neq \bar{p} } {f}(A \bar{p} + Av) {f}(A \bar{p} + Av) \, dv \, dm_{G_d}(A).
\end{align*}
To compute the second integral, we set $\bar{a}= \bar{p} - \bar{q}$, $\bar{b}= \bar{q}$  and use non-negativity of $f$ (which allows interchange of integration and summation) to get 
\begin{align*}
    &\int_{\mathcal{F}} \sum_{ \bar{p} } \sum_{ \bar{q} \neq \bar{p} } \int_{[0,1]^d}  {f}(A \bar{p} + Av) {f}(A \bar{p} + Av) \, dv \, dm_{G_d}(A) \\
    &= \int_{\mathcal{F}} \sum_{ \bar{a} \neq 0 }  \int_{[0,1]^d} \sum_{ \bar{b} }  {f}(A \bar{a} + A\bar{b}  + Av) {f}(A \bar{b} + Av) \, dv \, dm_{G_d}(A) \\
     &= \int_{\mathcal{F}} \sum_{ \bar{a} \neq 0 }  \int_{\R^d}  {f}(A \bar{a}  + v) {f}( v) \, dv \, dm_{G_d}(A) \\
      &= \int_{\X_d} \sum_{ u \in \Lambda \setminus \{0\} }  \int_{\R^d}  {f}(u  + v) {f}( v) \, dv \, dm_{\X_d}(\Lambda) \\
     &= \int_{\R^d}  \int_{\R^d}  {f}(w  + v) {f}( v) \, dv \, dw \\
     &= \left(\int_{\R^{d}}f \right)^2,
\end{align*}
where the second last equality holds by Siegel Mean Value Theorem \cite{sie}. This proves the theorem.
\end{proof}

\subsubsection{Non-divergence estimates of Siegel Transforms} \hfill \\

We retain the notation from Section \ref{Aff Intro}. Given 
$$
0<w_1,\ldots, w_m<n \text{ and }  w_1+\ldots+w_m=n,
$$
we denote by $a$ the self-map on $\X$ induced by $a=(a',0)$ where
\begin{equation}
a'={diag}(e^{w_1},\ldots,e^{w_m},e^{-1},\ldots,e^{-1}).
\end{equation}
Our goal in this subsection is to analyze escape of mass for the submanifolds $a^s\Y$ and 
bound the Siegel transforms $\hat f(a^s y)$ for $y\in \Y$. We have the following Proposition.

\begin{prop}
\label{Siegel Aff 1}
    There exists $\kappa>0$ such that for every $L \geq 1$ and $s \geq \kappa \log L$,
    \begin{equation*}
        \mu_{\Y}(\{y \in \Y : \alpha(a^sy) \geq L\}) \ll_p L^{-p} \text{    for all } p< m+n.
    \end{equation*}
\end{prop}
\begin{proof}
    It is clear that
    \begin{align*}
       \mu_{\Y}(\{y \in \Y : \alpha(a^sy) \geq L\}) &= \mu_{\Y}(\{y \in \Y : \alpha_d(\pi(a^sy)) \geq L\}) \\
        &= \mu_{\Y}(\{y \in \Y : \alpha_{d}((a')^s \pi(y)) \geq L\}) \\
        &=\mu_{\Y}(\pi^{-1}(\{z \in \Y_d : \alpha_{d}((a')^s z) \geq L\})) \\
        &= \mu_{\Y_d}(\{z \in \Y_d : \alpha_{d}((a')^s z) \geq L\}). 
    \end{align*}
   Now the proof follows from the Proposition. 
\begin{prop}[\cite{BG}, Prop. 4.5]
\label{Siegel Hom 1}
    There exists $\kappa>0$ such that for every $L \geq 1$ and $s \geq \kappa \log L$,
    \begin{equation*}
        \mu_{\Y_d}(\{z \in \Y_d : \alpha_{d}((a')^s z) \geq L\}) \ll_p L^{-p} \text{    for all } p< m+n.
    \end{equation*}
\end{prop}
   
\end{proof}

\begin{prop} 
\label{siegel main aff}
    Let $f$ be a bounded measurable function on $\R^{m+n}$ with compact support contained in the open set $\{(x_{m+1}, \ldots, x_{m+n}) \neq 0\}$. Then for $i=1,2$, we have
    \begin{equation}
       \underset{s \geq 0}{\sup} \| \hat{f} \circ a^s\|_{L^i(\Y)} < \infty.
    \end{equation}
\end{prop}

\begin{proof}
     We note that there exists $0< \upsilon_1 < \upsilon_2 $ and $\vartheta >0$ such that the support of $f$ is contained in the set 
    \begin{equation}
        \{(\bar{x},\bar{y}) \in \R^{m+n} : \upsilon_1 \leq \|\bar{y}\| \leq \upsilon_2, |x_i| \leq \vartheta\|\bar{y}\|^{-w_i}, i=1,\ldots, n\},
    \end{equation}
    and without loss of generality we may assume that $f$ is characteristic function of above set. We recall that $\Y$ can be identified with the collection of lattices $$\{ \Lambda_{u,v} :\bar{u} =(u_{ij}: 1 \leq i \leq m, 1 \leq j \leq n) \in [0,1)^{mn}, v \in [0,1)^m\}.$$ We set $\bar{u}_i := (u_{i1}, \ldots, u_{in}).$ Then by the definition of the Siegel transform and the fact that $f(0_d)=0$, we have

    \begin{align}
        \hat{f}(a^s\Lambda_{u,v}) = \sum_{(\bar{p},\bar{q}) \in \Z^{m+n} } f\left( e^{w_1 s}(p_1 + \langle \bar{u}_1, \bar{q} \rangle + v_1) ,\ldots, e^{w_ms }(p_m  + \langle \bar{u}_m, \bar{q}\rangle + v_m ) , e^{-s}\bar{q}  \right).
    \end{align}
    If we  denote by $\chi'_M$ the characteristic function of the interval $[-M,M]$, we get
    \begin{align}
        \hat{f}(a^s\Lambda_{u,v}) &= \sum_{  \upsilon_1 e^s \leq \|\bar{q} \| \leq \upsilon_2 e^s} \sum_{\bar{p} \in \Z^m}  \prod_{i=1}^m \chi'_{\vartheta\|\bar{q} \|^{-w_i}} \left(p_i + \langle \bar{u}_i, \bar{q} \rangle + v_i   \right)  \nonumber \\ 
        \label{aff int open 1}
        &= \sum_{  \upsilon_1 e^s \leq \|\bar{q} \| \leq \upsilon_2e^s}  \prod_{i=1}^m  \left( \sum_{p_i \in \Z}  \chi'_{\vartheta\|\bar{q} \|^{-w_i}} \left(p_i + \langle \bar{u}_i, \bar{q} \rangle + v_i   \right)  \right). 
    \end{align}
    We now define a norm $\|.\|_{n+1}$ on $\R^{n+1}$ as $$\|(x,y)\|_{n+1} = \|x\|+ |y|, \text{  for all } x \in \R^n, y \in \R.$$ Choose $t>0$ large enough so that $(2^{1/w_i}-1)\upsilon_1 e^t >1$ for all $i= 1 , \ldots,m$ and $\upsilon_2 e^t > 1$. Then for all $s \geq t$, we have %$ \|(\bar{q}, 1)\|_{m+1} < 2 \|\bar{q}\|$ for all $\bar{q} \in \Z^m$ satisfying $\upsilon_1 e^s \leq \|\bar{q} \| \leq \upsilon_2e^s.$ Thus, for $\bar{q} \in \Z^m$ satisfying $\upsilon_1 e^s \leq \|\bar{q} \| \leq \upsilon_2e^s$, we have
    \begin{align*}
         \chi'_{\vartheta\|\bar{q} \|^{-w_i}} \left(p_i + \langle \bar{u}_i, \bar{q} \rangle + v_i   \right) &\leq  \chi'_{ 2 \vartheta\|(\bar{q},1) \|_{n+1}^{-w_i}} \left(p_i + \langle \bar{u}_i, \bar{q} \rangle + v_i   \right)
    \end{align*}
   for all $\bar{q} \in \Z^m$ satisfying $\upsilon_1 e^s \leq \|\bar{q} \| \leq \upsilon_2e^s.$ With slight abuse of notation, the usual inner product on $\R^{n+1}$ will also be denoted by $\langle ., . \rangle$. Then, we have  
   \begin{align*}
       \hat{f}(a^s\Lambda_{u,v}) &= \sum_{  \upsilon_1 e^s \leq \|\bar{q} \| \leq \upsilon_2e^s}  \prod_{i=1}^m  \left( \sum_{p_i \in \Z}  \chi'_{\vartheta\|\bar{q} \|^{-w_i}} \left(p_i + \langle \bar{u}_i, \bar{q} \rangle + v_i   \right)  \right)  \\
       &\leq \sum_{  \upsilon_1 e^s \leq \|\bar{q} \| \leq \upsilon_2e^s}  \prod_{i=1}^m  \left( \sum_{p_i \in \Z}  \chi'_{ 2 \vartheta\|(\bar{q},1) \|_{n+1}^{-w_i}} \left(p_i + \langle \bar{u}_i, \bar{q} \rangle + v_i   \right)  \right) \\
       & = \sum_{  \upsilon_1 e^s \leq \|\bar{q} \| \leq \upsilon_2e^s}  \prod_{i=1}^m  \left( \sum_{p_i \in \Z}  \chi'_{ 2 \vartheta\|(\bar{q},1) \|_{n+1}^{-w_i}} \left(p_i + \langle (\bar{u}_i, v_i), (\bar{q},1) \rangle   \right)  \right) \\
       &\leq \sum_{  \upsilon_1 e^s \leq \|(\bar{q},1)\|_{m+1}  \leq 2\upsilon_2e^s}  \prod_{i=1}^m  \left( \sum_{p_i \in \Z}  \chi'_{ 2 \vartheta\|(\bar{q},1) \|_{n+1}^{-w_i}} \left(p_i + \langle (\bar{u}_i, v_i), (\bar{q},1) \rangle   \right)  \right) \\
       &\leq \sum_{ \substack{ \bar{q} \in \Z^{n+1} \\  \\ \upsilon_1 e^s \leq \| \bar{q}\|_{m+1}  \leq 2\upsilon_2e^s }}  \prod_{i=1}^m  \left( \sum_{p_i \in \Z}  \chi'_{ 2 \vartheta\|\bar{q} \|_{m+1}^{-w_i}} \left(p_i + \langle (\bar{u}_i, v_i), \bar{q} \rangle   \right)  \right).
   \end{align*}
Thus, we have 
\begin{align}
\label{aff 2 1}
    \|\hat{f} \circ a^s\|_{L^1(\Y)} &\leq \sum_{ \substack{ \bar{q} \in \Z^{n+1} \\  \\ \upsilon_1 e^s \leq \| \bar{q}\|_{m+1}  \leq 2\upsilon_2e^s }}  \prod_{i=1}^m  \left( \sum_{p_i \in \Z} \int_{[0,1]^{n+1}} \chi'_{ 2 \vartheta\|\bar{q} \|_{m+1}^{-w_i}} \left(p_i + \langle\bar{w}_i, \bar{q} \rangle   \right)  \, d \bar{w}_i \right)
\end{align}
and 
 \begin{align}
        &\|\hat{f} \circ a^s\|_{L^2(\Y)} = \int_{\Y} \hat{f}(a^sy) \hat{f}(a^sy) \, d\mu_{\Y}(y) \nonumber \\
       \label{aff 2 2}
        &\leq \sum_{ \substack{ \bar{q}, \bar{l} \in \Z^{n+1} \\  \\ \upsilon_1 e^s \leq \| \bar{q}\|_{m+1}, \leq 2\upsilon_2e^s \\  \\ \upsilon_1 e^s \leq \| \bar{l}\|_{m+1}, \leq 2\upsilon_2e^s }}  \prod_{i=1}^m \left(  \sum_{p_i, r_i \in \Z}  \int_{[0,1]^{n+1}} \chi'_{\frac{2 \vartheta}{\|\bar{q} \|_{m+1}^{w_i}}} \left(p_i + \langle \bar{w}_i, \bar{q} \rangle  \right)  \chi'_{\frac{2 \vartheta}{\|\bar{l} \|_{m+1}^{w_i}}} \left(r_i + \langle \bar{w}_i, \bar{l} \rangle  \right)  d\bar{w_i} \right).
\end{align}
Note that the expression \eqref{aff 2 1} is the same as that in equation (4.3) of \cite{BG}. The boundedness of the latter expression, independent of $s$, was proved in Proposition 4.6 of \cite{BG}. Similarly, the expression \eqref{aff 2 2} is the same as the equation obtained in Proposition 4.8 of \cite{BG}, and the boundedness of the latter expression, independent of $s$, is the content of the same Proposition.
\end{proof}

\subsubsection{Truncated Siegel Transforms} \label{Aff Truncated Siegel}\hfill \\

We define, for a bounded function $f: \R^{m+n} \rightarrow \R$ with compact support, the \textit{truncated Siegel transform} of $f$ as $$\hat{f}^{(L)} := \hat{f}. \eta_L,$$ where $\eta_L$ is defined below in Lemma \ref{aff trun 1}. We also record some basic properties of this transform in Lemma \ref{aff trun 2}.
\begin{lem}
\label{aff trun 1}
    For every $c>1$, there exists a family $(\eta_L)$ in $\C_c^{\infty}(\X)$ satisying
    \begin{equation*}
        0\leq \eta_L \leq 1,\ \ \eta_L= 1 \text{ on } \{\alpha \leq c^{-1}L\}, \ \ \eta_L =0 \text{ on } \{ \alpha > cL\}, \ \ \|\eta_L\|_{C^k} \ll 1
    \end{equation*}
\end{lem}
\begin{lem}
\label{aff trun 2}
    For $f \in \C_c^{\infty}(\R^{m+n})$, the truncated Siegel transform $\hat{f}^{(L)}:= \hat{f}. \eta_L$ is in $\C-C^{\infty}(\X)$, and it satisfies
    \begin{align*}
      \|\hat{f}^{(L)}\|_{L^p(\X)} \leq \|\hat{f}\|_{L^p(\X)} &\ll_{supp(f),p} \|f\|_{C^0} \text{    for all } p<m+n  \\
      \|\hat{f}^{(L)}\|_{C^0} &\ll_{supp(f)} L \|f\|_{C^0}  \\
      \|\hat{f}^{(L)}\|_{C^k} &\ll_{supp(f)} L \|f\|_{C^k}  \\
      \|\hat{f}-\hat{f}^{(L)}\|_{L^1(\X)} &\ll_{supp(f),\tau} L^{-\tau}\|f\|_{C^0} \text{    for all } \tau<m+n-1 \\
      \|\hat{f}-\hat{f}^{(L)}\|_{L^2(\X)} &\ll_{supp(f),\tau} L^{-(\tau-1)/2}\|f\|_{C^0} \text{    for all } \tau<m+n-1 
    \end{align*}
    Moreover, the implied constants are uniform when $supp(f)$ is contained in a fixed compact set.
\end{lem}
The proofs of Lemmas \ref{aff trun 1}, \ref{aff trun 2} are exact replicas of the proofs of Lemmas 4.11 and 4.12 of \cite{BG} respectively, so we have not repeated them.

\subsection{CLT for Smooth Siegel Transforms}
\begin{thm}
    Assume $f \in \C_c^{\infty}(\R^{m+n})$ satisfies $f \geq 0$ and $supp(f) \subset \{(x_{m+1}, \ldots, x_{m+n}) \neq 0\}$. Assume  $m \geq 2$. Consider the sequence of averages
    \begin{equation*}
        F_M(y):= \frac{1}{\sqrt{M}} \sum_{s=0}^{M-1}(\hat{f}(a^s)-  \mu_{\Y}(\hat{f} \circ a^s)) \text{     with } y \in \Y
    \end{equation*}
    where $a= (diag(e^{w_1},\ldots,e^{w_m},e^{-1},\ldots,e^{-1}),0)$ . Then the variance
    \begin{equation*}
        \sigma_f^2 := \sum_{s= -\infty}^{\infty} \int_{\R^d} f(a^sx) f(x) \, dx
    \end{equation*}
    is finite, and for every $\xi \in \R$,
    \begin{equation*}
        \mu_{\Y}(\{ y \in \Y : F_M(y) < \xi\}) \rightarrow Norm_{\sigma_f}(\xi)
    \end{equation*}
    as $N\rightarrow \infty$.
\end{thm}
\begin{proof}
    All the steps given in the proof of [Thm 5.1, \cite{BG}] apply here directly to give the result here. We have tried to keep the same notations to aid the reader. The only change required is the computation of the right hand side equation (5.36) of \cite{BG}, i.e, the computation of the variance. We compute it here. Equation (5.36) of \cite{BG} gives
    \begin{equation*}
        \|F_M^{(L)}\|_{L^2(\Y)} \rightarrow \Theta_{\infty}(0) + 2 \sum_{s=1}^{\infty} \Theta_{\infty}(s) = \sum_{s= -\infty}^{\infty} \Theta_{\infty}(s)
    \end{equation*}
    as $M \rightarrow \infty$, where
    \begin{align*}
        \Theta_{\infty}(s) &:= \int_{\X}(\hat{f} \circ a^s) \hat{f}\, d\mu_{\X} -\mu_{\X}(\hat{f})^2 \\
        &=  \frac{1}{2} \left(\int_{\X}(\hat{f} \circ a^s + \hat{f})^2\, d\mu_{\X} - 2\int_{\X}( \hat{f})^2\, d\mu_{\X}  -2\mu_\X(\hat f)^2 \right) \\
    &= \frac{1}{2} \left(\left( \int_{\R^d}(f \circ a^s + f) \right)^2 + \int_{\R^d}(f \circ a^s + f )^2 - 2\left(  \left( \int_{\R^d} f \right)^2 + \int_{\R^d}( f )^2\right) -2 \left( \int_{\R^d} f \right)^2  \right)\\
    &= \int_{\R^{d}} f(a^s z) f(z)\, dz,
    \end{align*}
    where the third equality follows from Proposition \ref{aff S R}.
    Thus,
    \begin{equation}
    \label{var1}
     \sum_{s= -\infty}^{\infty} \Theta_{\infty}(s) = \sum_{s= -\infty}^{\infty}   \int_{\R^{d}} f(a^s x) f(x))\, dx.  
    \end{equation}
    Now, we show that sum in \eqref{var1} is finite. This is clear since $\Theta_{\infty}(s)$ is zero for all but finitely many values. To see this, we represent points $\bar{z} \in \R^{m+n}$ as $\bar{z}= (\bar{x},\bar{y})$ with $\bar{x} \in \R^m$ and $\bar{y} \in \R^n$.  Since, $f$ is bounded, and the compact support of $f$ is contained in $ \{ \bar{y} \neq 0\}$, we may assume without loss of generality that $f$ is the characteristic function of the set
    \begin{equation*}
        \{ (\bar{x},\bar{y}) \in \R^{m+n} : \upsilon_1 \leq \|\bar{y}\| \leq \upsilon_2, |x_i| \leq \vartheta \|\bar{y}\|^{-w_i} \text{, } i=1, \ldots,m \}
    \end{equation*}
    with $0< \upsilon_1 < \upsilon_2$ and $\vartheta >0$. Then
    \begin{equation}
    \label{fin}
     \Theta_{\infty}(s) =    \Theta_{\infty}(|s|) = \int_{\R^{m+n}} f(a^{|s|} \bar{z}) f(\bar{z}))\, d\bar{z} = \underset{\|\bar{y}\| \in [\upsilon_1,\upsilon_2] \cap [e^{|s|}\upsilon_1, e^{|s|}\upsilon_2]}{\int} \prod_{i=1}^{m} \left( \frac{2\vartheta}{\|\bar{y}\|^{-w_i}}\right) \, d\bar{y}.
    \end{equation}
    Now, $[\upsilon_1,\upsilon_2] \cap [e^{|s|}\upsilon_1, e^{|s|}\upsilon_2] \neq \phi$ only if $ |s| \leq \log{\vartheta_2}- \log{\vartheta_2}$.
    Hence, the sum \eqref{var1} is indeed finite. The rest of the proof now follows by resuming the argument from Section 5.3 in \cite{BG}.
\end{proof}

\subsection{CLT for counting functions and the proof of Theorem \ref{CLT2}} \hfill \\
We will need a little more notation before we begin the proof of Theorem \ref{CLT2}.

Let 
\begin{equation}
    \label{eq:def Omega}
    \Omega_T := \{(\bar{x},\bar{y}) \in \R^{m+n}:1\leq \|\bar{y}\| < T, |x_i|< \vartheta_i \|\bar{y}\|^{-w_i}, i=1,\ldots,m\}.
\end{equation}
Then, it is clear that $$\Delta_T(u,v)= |\Lambda_{u,v} \cap \Omega_T| + \mathcal{O}(1) \text{,   for } T>0.  $$ With $a= ({diag}(e^{w_1},\ldots,e^{w_m},e^{-1},\ldots,e^{-1}),0)$,  we note that for any integer $M \geq 1$, $$\Omega_{e^M} =  \sqcup_{s=0}^{M-1}a^{-s} \Omega_e, $$ and thus $$|\Delta_{u,v} \cap \Omega_{e^M}| = \sum_{s=0}^{M-1} \hat{\chi}(a^s \Lambda_{u,v}),$$ where $\chi$ denotes the characteristic function of the set $\Omega_e$. Define $$F_M:= \frac{1}{\sqrt{M}} \sum_{s=0}^{M-1} \left( \hat{\chi} \circ a^s - \mu_{\Y}(\hat{\chi} \circ a^s) \right).$$ We approximate $\chi$ by a family of non-negative functions $f_{\e} \in \Cc(\R^{m+n})$ whose supports are contained in an $\e$-neighbourhood of the set $\Omega_e$, and 
$$
\chi \leq f_\e \leq 1,\  \|f_\e -\chi\|_{L^1(\R^{m+n})} \ll \e, \ \|f_\e - \chi \|_{L^2(\R^{m+n})} \ll \e^{1/2}, \ \|f_\e\|_{C^k} \ll \e^{-k}.
$$
This approximation allows us to construct smooth approximation of the Siegel transform $\hat \chi$ in the following sense.

\begin{prop}\label{needednext}
        For every $s \geq 0$, 
        \begin{equation*}
            \int_{\Y} \left| \hat{f_{\epsilon}} \circ a^s - \hat{\chi} \circ a^s \right| \, d\mu_{\Y} \ll \epsilon + e^{-s}.
        \end{equation*}
    \end{prop}
    \begin{proof}

        We observe that there exists $\vartheta_{i}(\epsilon)> \vartheta_i$ such that $\vartheta_i(\epsilon)= \vartheta_i + \mathcal{O}(\epsilon)$ and $f_{\epsilon} \leq \chi_{\epsilon}$, where $\chi_{\epsilon}$ denotes the characteristic function of the set 
        \begin{equation}
            \{(\bar{x},\bar{y}) \in \R^{m+n}: 1- \epsilon \leq \|\bar{y}\| \leq e +\epsilon, |x_i| < \vartheta_{i}(\epsilon) \|\bar{y}\|^{-w_i} \text{  for } i= 1,\ldots, m\}.
        \end{equation}

        Then, it follows that
        \begin{equation*}
            |\hat{f_{\epsilon}}(a^s \Lambda) - \hat{\chi}(a^s\Lambda)| = \sum_{v \in \Lambda \setminus \{0\}} (f_{\epsilon}(a^sv)- \chi(a^sv)) \leq \sum_{v \in \Lambda \setminus \{0\}} (\chi_{\epsilon}(a^sv)- \chi(a^sv)) .
        \end{equation*}
        It is clear that $\chi_{\epsilon}- \xi$ is bounded by the sum $\chi_{1,\epsilon} + \chi_{2,\epsilon} + \chi_{3,\epsilon}$ of the characteristic functions of the sets
        \begin{eqnarray*}
        &\{& (\bar{x},\bar{y}) \in \R^{m+n} : 1-\epsilon \leq \|\bar{y}\| \leq 1,  |x_i| < \vartheta_{i}(\epsilon) \|\bar{y}\|^{-w_i} \text{  for } i= 1,\ldots, m\} \\
        &\{&  (\bar{x},\bar{y}) \in \R^{m+n} : e \leq \|\bar{y}\| \leq  e+ \epsilon,  |x_i| < \vartheta_{i}(\epsilon) \|\bar{y}\|^{-w_i} \text{  for } i= 1,\ldots, m\} \\
        &\{&  (\bar{x},\bar{y}) \in \R^{m+n} : 1 \leq \|\bar{y}\| \leq  e,  |x_i| < \vartheta_{i}(\epsilon) \|\bar{y}\|^{-w_i} \text{  for } i= 1,\ldots, m, |x_j|\geq \vartheta_j \|\bar{y}\|^{-w_j} \text{  for some } j\}.
        \end{eqnarray*}
        respectively. In particular, we obtain that
        \begin{equation*}
            \hat{f}_{\epsilon}(a^s \Lambda) - \hat{\chi}(a^s\Lambda) \leq \bar{\chi}_{1,\epsilon}(a^s \Lambda) + \bar{\chi}_{2,\epsilon}(a^s \Lambda) + \bar{\chi}_{3,\epsilon}(a^s \Lambda).
        \end{equation*}
        Hence, it remains to show that for $j=1,2,3$, 
        \begin{equation*}
            \int_{\Y} (\hat{\chi}_{j, \epsilon} \circ a^s) \, d\mu_{\Y} \ll \epsilon + e^{-s}.
        \end{equation*}
        \begin{align*}
            \int_{\Y} (\hat{\chi}_{1, \epsilon} \circ a^s) \, d\mu_{\Y} &= \underset{(1- \e)e^s \leq \|\bar{q}\| \leq e^s}{ \sum} \prod_{i=1}^{m} \left( \int_{[0,1]^n}  \left( \underset{\bar{p_i} \in \Z}{ \sum} \int_{[0,1]}  \chi_{\vartheta_i(\e)\|\bar{q}\|^{-w_i}}'(p_i+ \langle \bar{u}_i, \bar{q} \rangle + v_i) \, dv_i \right) \, d\bar{u_i} \right) \nonumber \\
            &=\underset{(1- \e)e^s \leq \|\bar{q}\| \leq e^s}{ \sum} \prod_{i=1}^{m} \left( \int_{[0,1]^n}  \left(  \int_{\R}  \chi_{\vartheta_i(\e)\|\bar{q}\|^{-w_i}}'( \langle \bar{u}_i, \bar{q} \rangle + v_i) \, dv_i \right) \, d\bar{u_i} \right) \nonumber \\
             &= \underset{(1- \e)e^s \leq \|\bar{q}\| \leq e^s}{ \sum} \prod_{i=1}^{m} \left( \int_{[0,1]^n} (2\vartheta_i(\e)\|\bar{q}\|^{-w_i}) \, d\bar{u}_i \right) \nonumber \\
             &\ll \underset{(1- \e)e^s \leq \|\bar{q}\| \leq e^s}{ \sum} \|\bar{q}\|^{-n} \\
             &\ll e^{-ns} |\{ \bar{q} \in \Z^n: (1-\epsilon)e^s \leq \|\bar{q}\| \leq e^s \}|
         \end{align*}
        The number of integral points in the region $\{ (1-\epsilon)e^s \leq \|\bar{q}\| \leq e^s \}$ can be estimated in terms of its volume. Namely, there exists $r>0$ (depending only on the norm) such that 
        \begin{equation*}
           |\{ \bar{q} \in \Z^n: (1-\epsilon)e^s \leq \|\bar{q}\| \leq e^s \}| \ll  |\{ \bar{y} \in \R^n: (1-\epsilon)e^s -r \leq \|\bar{q}\| \leq e^s +r \}|.
        \end{equation*}
        Hence, 
        \begin{eqnarray*}
            \int_{\Y} (\hat{\chi}_{1, \epsilon} \circ a^s) \, d\mu_{\Y} &\ll& e^{-ns} ((e^s +r)^n- ((1-\epsilon)e^s-r)^n) \\
            &=& (1 +re^{-s})^n- ((1-\epsilon)-re^{-s})^n \\
            &\ll& \epsilon + e^{-s}.
        \end{eqnarray*}
        The integral $\hat{\chi}_{2,\epsilon} \circ a^s$ can be estimated similarly.\\ 

       A similar calculations of the integral over $\hat{\chi}_{3,\epsilon} \circ a^s$ give
        \begin{align*}
           \int_{\Y} (\hat{\chi}_{3, \epsilon} \circ a^s) \, d\mu_{\Y} &=
           \underset{e^s \leq \|\bar{q}\| \leq e^{s+1}}{ \sum} \prod_{i=1}^{m} \left( \int_{[0,1]^n}  \left( \underset{\bar{p_i} \in \Z}{ \sum} \int_{[0,1]}  \chi_{\vartheta_i(\e)\|\bar{q}\|^{-w_i}}'(p_i+ \langle \bar{u}_i, \bar{q} \rangle + v_i) \, dv_i \right) \, d\bar{u_i} \right) \\
           & - \underset{e^s \leq \|\bar{q}\| \leq e^{s+1}}{ \sum} \prod_{i=1}^{m} \left( \int_{[0,1]^n}  \left( \underset{\bar{p_i} \in \Z}{ \sum} \int_{[0,1]}  \chi_{\vartheta_i\|\bar{q}\|^{-w_i}}'(p_i+ \langle \bar{u}_i, \bar{q} \rangle + v_i) \, dv_i \right) \, d\bar{u_i} \right)
           \\ &=  \underset{e^s \leq \|\bar{q})\| \leq e^{s+1}}{ \sum}  2^m \left(\prod_{i=1}^m\vartheta_{i}(\epsilon) - \prod_{i=1}^m \vartheta_i \right)   \|\bar{q}\|^{-n} \\
           &\ll \epsilon \left(\underset{e^s \leq \|\bar{q}\| \leq e^{s+1}}{ \sum}  \|\bar{q}\|^{-n} \right) \\
           &\ll \epsilon.
        \end{align*}
        This completes the proof of the proposition.
\end{proof}

\begin{thm}
\label{CLT2'}
    If $m \geq 2$, then for every $\xi \in \R$, 
    \begin{equation*}
        \mu_{\Y}(\{ y \in \Y : F_M(y) < \xi\}) \rightarrow Norm_{\sigma}(\xi)
    \end{equation*}
    as $M\rightarrow \infty$, where 
    \begin{equation*}
    \sigma^2= 2^m \vartheta_1\vartheta_2 \ldots \vartheta_m \omega_n \text{ with } \omega_n := \int_{S^{n-1}} d\bar{z}.
\end{equation*}
\end{thm}
\begin{proof}
    The same steps given in the proof of [Thm 6.1, \cite{BG}] work here. The only changes needed are that we need to reprove [Prop 6.2 \cite{BG}] in our case, and we have done this is Proposition \ref{needednext} above. Also, we need to compute the variance $\sigma^2$. As in proof of Theorem 6.1 of \cite{BG}, we get that
        \begin{equation}
            \label{aff eq: main variance}
            \left\|\Tilde{F}^{(\e,L)}_M\right\|_{L^2(\Y)}^2 =
            \Theta^{(\e)}_\infty(0)+2\sum_{s=1}^{K-1}\Theta^{(\e)}_\infty (s)+o(1),
        \end{equation}
    where 
    $$
    \Theta^{(\e)}_\infty(s):=\int_{\X} (\hat f_\e\circ a^s)\hat f_\e\, d\mu_\X -\mu_\X(\hat f_\e)^2.
    $$
    Using Proposition \ref{aff S R}, we get that
\begin{align*}
    \Theta^{(\e)}_\infty (s) &= \int_{\X} (\hat f_\e\circ a^s)\hat f_\e\, d\mu_\X -\mu_\X(\hat f_\e)^2 \\
    &=  \frac{1}{2} \left(\int_{\X}(\hat{f}_\e \circ a^s + \hat{f}_\e)^2\, d\mu_{\X} - 2\int_{\X}( \hat{f}_\e)^2\, d\mu_{\X}  -2\mu_\X(\hat f_\e)^2 \right) \\
    &= \frac{1}{2} \left(\left( \int_{\R^d}(f_\e \circ a^s + f_\e) \right)^2 + \int_{\R^d}(f_\e \circ a^s + f_\e )^2 - 2\left(  \left( \int_{\R^d} f_\e \right)^2 + \int_{\R^d}( f_\e )^2\right) -2 \left( \int_{\R^d} f_\e \right)^2  \right)\\
    &= \int_{\R^{m+n}} f_\e(a^s z) f_\e(z)\, dz.
\end{align*}
Now, for $s \in \Z$ define 
\begin{equation*}
    \Theta_{\infty}(s)= \int_{\R^{m+n}} \chi(a^s z) \chi(z) \, dz.
\end{equation*}
It is clear that 
\begin{align}
    | \Theta^{(\e)}_\infty (s) - \Theta_{\infty}(s)| &\leq \int_{\R^{m+n}} |f_\e(a^s z)- \chi(a^s z )| |f_\e(z)|\, dz + \int_{\R^{m+n}} |\chi(a^s z)| |f_\e(z)- \chi(z)|\, dz \\
    &\leq \|f_\e - \chi\|_{L^2(\R)} \|f_\e\|_{L^2(\R)} + \|\chi\|_{L^2(\R)} \|f_\e- \chi\|_{L^2(\R)} \\
    &\ll \e^{1/2}.
\end{align}
Also, note that $\Theta_{\infty}^{(\e)}(s)$ is zero for all but finitely many values. To see this, first note that by construction of $f_\e$, there exists $0< \upsilon_1 < \upsilon_2$ and $\vartheta >0$ such that the characteristic function $\psi$ of the set
    \begin{equation*}
        \{ (\bar{x},\bar{y}) \in \R^{m+n} : \upsilon_1 \leq \|\bar{y}\| \leq \upsilon_2, |x_i| \leq \vartheta \|\bar{y}\|^{-w_i} \text{, } i=1, \ldots,m \}
    \end{equation*}
    satisfy $ f_\e(\bar z ) \leq \psi(\bar{z}) $ for all $\e >0$. Then
    \begin{align*}
     \Theta_{\infty}^{(\e)}(s) &=    \Theta_{\infty}^{(\e)}(|s|) \\ &= \int_{\R^{m+n}} f_\e(a^{|s|} \bar{z}) f_\e(\bar{z})\, d\bar{z} \\ & \leq \int_{\R^{m+n}} \psi(a^{|s|} \bar{z}) \psi(\bar{z}))\, d\bar{z} \\  &= \underset{\|\bar{y}\| \in [\upsilon_1,\upsilon_2] \cap [e^{|s|}\upsilon_1, e^{|s|}\upsilon_2]}{\int}  2^m\vartheta^m\|\bar{y}\|^{-n} \, d\bar{y}.
    \end{align*}
    Now, $[\upsilon_1,\upsilon_2] \cap [e^{|s|}\upsilon_1, e^{|s|}\upsilon_2] \neq \phi$ only if $ |s| \leq \log{\vartheta_2}- \log{\vartheta_2}$.
   Thus $\Theta_{\infty}^{(\e)}(s)$ is zero for all but finitely many values, uniformly on $\e$.

 Also, by explicit computation it is easy to see that $\Theta_{\infty}(s)=0$ for all $s \neq 0.$ For $s=0$, we have $\Theta_{\infty}(0)$ equals the measure of the set $\Omega_e$, defined as in \eqref{eq:def Omega}. Hence, $\Theta_{\infty}(0)= 2^{m+1} \vartheta_1\vartheta_2 \ldots \vartheta_m \omega_n \text{ with } \omega_n := \int_{S^{n-1}}  d\bar{z}$. 

 Combining the above statements, we get that since $\e(M) \rightarrow 0$ as $M \rightarrow \infty$, so  
 \begin{align*}
    \Theta^{(\e)}_\infty(0)+2\sum_{s=1}^{K(M)-1}\Theta^{(\e)}_\infty (s) \rightarrow 2^m \vartheta_1\vartheta_2 \ldots \vartheta_m \omega_n,
\end{align*}
which proves that 
\begin{equation}
    \label{aff variance cal 1}
    \Theta^{(\e)}_\infty(0)+2\sum_{s=1}^{K-1}\Theta^{(\e)}_\infty (s) \rightarrow 2^m \vartheta_1\vartheta_2 \ldots \vartheta_m \omega_n
\end{equation} as $M \rightarrow \infty$.
Now, resuming the steps in proof of Theorem 6.1 of \cite{BG} will give us the proof.
\end{proof}

\begin{proof}[Proof of Theorem \ref{CLT2}]

    The only step that require change is Lemma 6.3 of \cite{BG}. We need to compute  $\sum_{s=0}^{M-1} \int_{\Y} \hat{\chi}(a^sy) \, d\mu_{\Y}(y)$. 
We observe that
$$
\sum_{s=0}^{M-1}\int_{\Y} \hat \chi(a^s y)\,d\mu_\Y(y)
=\int_{\Y} \hat \Xi_M( y)\,d\mu_\Y(y),
$$
where $\Xi_M$ denotes the characteristic function of the set
$$
\left\{(\overline{x},\overline{y})\in \R^{m+n}:\, 1\le \|\overline{y}\|< e^M, \;\; |x_i|<\vartheta_i\, \|\overline{y}\|^{-w_i},\; i=1,\ldots,m\right\}.
$$
As in \eqref{aff int open 1}, we obtain 
\begin{align*}
\int_{\Y} \hat \Xi_M( y)\,d\mu_\Y(y) &=\sum_{1\le \|\overline{q}\|< e^M}   \prod_{i=1}^m \left(\int_{[0,1]^n} \left( \sum_{p_i\in \mathbb{Z}}  \int_{[0,1]}\chi'_{\vartheta\|\bar{q} \|^{-w_i}}\left(p_i+\left<\overline{u}_i,\overline{q}\right> + v_i\right) \, dv_i \right)\, d\overline{u}_i\right) \\
 &=\sum_{1\le \|\overline{q}\|< e^M}   \prod_{i=1}^m \left(\int_{[0,1]^n} \left(  \int_{\R}\chi'_{\vartheta\|\bar{q} \|^{-w_i}}\left(\left<\overline{u}_i,\overline{q}\right> + v_i\right) \, dv_i \right)\, d\overline{u}_i\right) \\
 &= \sum_{1\le \|\overline{q}\|< e^M}   \prod_{i=1}^m \left(\int_{[0,1]^n}  2 \vartheta_i\, \|\overline{q}\|^{-w_i} \right) \, dv_i \, d\overline{u}_i \\
&= 2^m \vartheta_1 \ldots\vartheta_m  \sum_{1\le \|\overline{q}\|< e^M} \|\bar{q}\|.
\end{align*}
Now, using that $\|\bar{y_2}\|^{-n}= \|y_1\|^{-n} + \mathcal{O}(\|\bar{y}_1\|^{-n-1})$ when $\|\bar{y_1} - \bar{y_2}\| \ll 1$, we deduce that $$\underset{1 \leq \|\bar{q}\| \leq e^M}{ \sum} \|\bar{q}\|^{-n} = \int_{1 \leq \|\bar{y}\| \leq e^M}   \|\bar{y}\|^{-n} \, d\bar{y} +\mathcal{O}(1),$$ and expressing the integral in polar co-ordinates, we obtain $$\int_{1 \leq \|\bar{y}\| \leq e^M}   \|\bar{y}\|^{-n} \, d\bar{y} = \int_{S^{n-1}} \int_{\|z\|^{-1}}^{\|z\|^{-1}e^M}  \|r \bar{z}\|^{-n} r^{n-1} \, dr \, d\bar{z}= \omega_nM +\mathcal{O}(1).$$ This proves the corresponding lemma. The remaining argument is identical.
\end{proof}

\section{Proof of the congruence case}
\subsection{Introduction} \label{Con Intro}  \hfill \\

This section is devoted to proving Theorem \ref{CLT3}. For this entire section, we again fix $m,n \in \N$ with $m \geq 2$ and $\|.\|$, a norm on $\R^n$. We also fix $\vartheta_1, \ldots, \vartheta_m > 0$ and $w_1, \ldots, w_m >0 $  satisfying $$w_1  + \ldots + w_m = n.$$  Also fix a vector $v =(v_1, \ldots,v_{m+n}) \in \Z^{m+n}$ and $N \in \N$ such that $N \geq 1$ and $\gcd(v,N)=1$. Let us denote $v':= (v_1, \ldots, v_m) \in \Z^m$ and $v''= (v_{m+1}, \ldots, v_{m+n})$. Then, define for $T>0$, $u \in \M_{m \times n}(\R)$
\begin{align*}
       \Delta_T(u) = |\{ (\bar{p},\bar{q}) \in \Z^m \times \Z^n: 0 < \|\bar{q}\| < T \ \text{ and } \ |p_i + L_{u}^{(i)}(q_1, \ldots, q_n)| < \vartheta_i \|\bar{q}\|^{-w_i} \\ \text{ for all } \ i=1, \ldots , m  \ \text{ and }
       (\bar{p}, \bar{q}) = v \mod N\}|.
\end{align*}
where $$L_{u}^{(i)} (x_1, \ldots, x_n) = \sum_{j=1}^n u_{ij}x_j, \text{    for all } i= 1, \ldots,m.$$

  Set $G \  := \ SL_{m+n}(\R)$ and $\Gamma:= \{ A \in SL_{m+n}(\Z) : A \V = \V \mod N\}$. Let $\X$ denote the space of affine unimodular lattice of the form  $A(\Z^{m+n} + \V / N)$ with $A \in G$. By definition, it is clear that we may consider $\X$ as a homogeneous space of $G$, under the covering map $$A \mapsto A(\Z^{m+n} + \V / N).$$ Since, the kernel of above map is $\Gamma$, we get that $\X \backsimeq G/\Gamma$. With this identification, $\X$ carries a natural transitive action of $G$, given by 
  \begin{equation*}
      A. \Lambda = A. (B(\Z^{m+n} + \V / N)) = (AB)(\Z^{m+n} + \V / N)
  \end{equation*}
  for all $\Lambda= B(\Z^{m+n} + \V / N) \in \X.$
   Note that $\Gamma \supset \Gamma(N)$, where $$\Gamma(N)= \left\{ A \in \SL_{m+n}(\Z): A = I \mod N\right\}.$$ It is easy to see that $\Gamma(N)$ is a finite index subgroup of $\SL_{m+n}(\Z)$, indeed it is a normal subgroup of $\SL_{m+n}(\Z)$ and $\SL_{m+n}(\Z)/ \Gamma(N) \simeq \SL_{m+n}(\Z/ N\Z)$. Thus, $\Gamma$ is also a finite index subgroup of $\SL_{m+n}(\Z)$, hence also a lattice in $\SL_{m+n}(\R)$. So, there exists $\mu_{\X}$, the G-invariant probability measure on $\X$. Let us denote by $m_G$ the bi-invariant Haar measure on G, normalized so that, such that fundamental region for action of $\Gamma $ on $G$ has measure 1. Denote by $U$ the subgroup
\begin{equation}
    U:= \left\{ \begin{pmatrix} I_m & u \\ 0 & I_{n} \end{pmatrix} : u \in \M_{m \times (n)}(\R) \right\} < G.
\end{equation}
Let $\Y := U(\Z^{m+n} + \V/N)  = \{ u. (\Z^{m+n}+ v/N): u \in U\} \subset \X$. %where elements of G act on $\R^{m+n}$ as usual linear maps, i.e,
%\begin{equation*}
 %   A.x= Ax \text{ for }A \in SL_{m+n}(\R).
%\end{equation*}
 We denote by $\mu_{\Y}$ the U-invariant probability measure on $\Y$ . Note that elements of $\Y$ look like
 \begin{equation}
 \label{defLambda3}
     \Lambda_{u}= \left\{ \left( \frac{Np_1 + v_1+ \langle \bar{u}_1, N\bar{q} + v'' \rangle }{N} ,\ldots, \frac{Np_m + v_m+ \langle \bar{u}_m, N\bar{q} + v'' \rangle }{N} , \frac{N\bar{q} +v'' }{N}\right) : (\bar{p}, \bar{q}) \in \Z^m \times \Z^n \right\} 
 \end{equation}
 where   $u \in \M_{m \times n}(\R)$, $\bar{u}_i$ denotes the $i$-th row of the matrix $\bar{u}$ and $\langle .,.\rangle$ denotes the standard inner product in $\R^n$.  \\
We set $d:= m+n$ and define  $\X_d= G/\SL_d(\Z)$ and $\Y_d = (U.\SL_d(\Z))/\SL_d(\Z).$ Denote by $\mu_{\X_d}$ the unique $G$-invariant probability measure on $\X_d$. Similarly denote by $\mu_{\Y_d}$, the unique $U$-invariant measure on $\Y_d$. Let us denote by $\pi$ the usual finite index covering map from $\X$ onto $\X_d$, which exists since $\Gamma$ is a finite index subgroup of $\SL_d(\Z)$. Then, it is clear that $\pi: (\X, \mu_{\X}) \rightarrow (\X_d, \mu_{\X_d})$ is a measure preserving map.  Also, it is clear that $\pi : \X \rightarrow \X_d$ maps $\Y$ onto $\Y_d$ and the restriction map $\pi_{|_{\Y}} : (\Y, \mu_{\Y}) \rightarrow (\Y_d,\mu_{\Y_d})$ is a measure preserving map. \\\\
An important observation here is that the fundamental domain for the action of $U \cap \Gamma(N)$ on $U$ is $\left\{ \begin{pmatrix} I_m & u \\ 0 & I_{n} \end{pmatrix} : u \in \M_{m \times n}([0,N)) \right\} \approx \M_{m \times n}([0,N))  .$ Thus, the restriction of the covering map $U \rightarrow \Y \simeq U/(U \cap \Gamma)$ to the space $\M_{m \times n}([0,N))$ is still a covering map (since $\Gamma \supset \Gamma(N)$). Thus if $m_U$ denotes the usual Lebesgue measure on $U \simeq \M_{m \times n}(\R)$, the map $$(\M_{m \times n}([0,N)), \frac{1}{N^{nm}} m_U) \rightarrow (\X, \mu_{\X})$$ is a measure preserving map. More precisely, for any $f \in L^1(\X)$ we have 
\begin{equation}
    \label{con covering preserves measure}
    \int_{\X} f(x)\, d\mu_{\X}(x)= \frac{1}{N^{nm}} \int_{\M_{m \times n}([0,N])} f(\Lambda_u) \, dm_U(u). 
\end{equation}
This observation will be used in several proofs ahead.

\subsection{Mixing of the $a^s \Y$ action on $\X$} \label{Con Mix Property} \vspace{0.3in} \hfill \\

Once again, we recall notation from \cite{BG3}. Fix positive weights $w_1, w_2, \ldots, w_{m+n}$ , satisfying $$\sum_{i=1}^{m} w_i = \sum_{i=m+1}^{m+n} w_i$$ and denote by $(a_t)$ the one parameter semi-subgroup $$a_t := diag(e^{w_1t}, \ldots, e^{w_{m}t}, e^{-w_{m+1}t}, \ldots, e^{-w_{m+n}t}) \ \text{, } t>0. $$

Again, we set
\begin{equation}
    b_t := diag(e^{t/m}, \ldots , e^{t/m}, e^{-t/n}, \ldots, e^{-t/n}), \ t>0
\end{equation}
to be the equal weight subgroup which coincides with $(a_{t})$ with the special choice of exponents $$w_1= \ldots= w_m = \frac{1}{m}, \ \ \ w_{m+1}= \ldots= w_{m+n}= \frac{1}{n}.$$

Fix an ordered basis ${Y_1,Y_2,\ldots, Y_r}$ of $Lie(G)$. Clearly, every monomial $Z= Y_1^{l_1}\ldots Y_r^{l_r} $ defines a differential operator by \eqref{diff hom}. We now define the norms $\|.\|_{L_k^2(\X)}$ and $\|.\|_{C^k}$ on $\C_c^{\infty} (\X)$ by equations \eqref{eq: def Lk} and \eqref{eq: def Ck} respectively. 
Analogous to equation \eqref{xi ineq 2}, one can easily show the existence of a constant $\xi = \xi(m,n,k)$ (which also depends on fixed choice of weights $w_1, \ldots, w_{m}$) such that 
 \begin{equation}
 \label{xi ineq 3}
     \|\phi \circ a^t\| \ll e^{\xi t} \|\phi\|_{C^k}, \text{    for all $t \geq 0$ and } \phi \in \C_c^{\infty}(\X), 
 \end{equation}
 where the suppressed constants are independent of $t$ and $\phi$.

 The following quantitative estimate on correlations of smooth functions on $\X$ is proved in \cite{BG3}.
\begin{thm}[\cite{BG3}, Thm. 1.1 ]
\label{Con Mixing}
There exists $\gamma > 0$ and $k \geq 1$ such that for all $\phi_1, \phi_2 \in \C_c^{\infty}(\X)$ and $g \in G$,
\begin{equation}
    \int_{\X} \phi_1(b_tx) \phi_2(x) \, d\mu_{\X}(x) = \left( \int_{\X} \phi_1(x) \, d\mu_{\X}(x)\right) \left( \int_{\X} \phi_2(x) \, d\mu_{\X}(x) \right) + \mathcal{O}\left( e^{-\gamma t} \|\phi_1\|_{\C^k}\|\phi_2\|_{\C^k} \right)
\end{equation}
where $b_t= diag(e^{t/m}, \ldots, e^{t/m}, e^{-t/n}, \ldots, e^{-t/n}). $
\end{thm}

\begin{rem}
    In [\cite{BG3}, Thm. 1.1], a more general result, namely quantitative multiple mixing for the action of a connected semi-simple Lie group $H$ with finite center on a homogeneous space $L/L^*$ is proved, under a suitable spectral gap hypothesis, where $L$ is a connected Lie group containing $H$ and $L^*$ is a lattice in $L$. The above statement in the case of 2-mixing has been proved much earlier and can also be found in [\cite{KS}, Cor. 3.2]. 
\end{rem}

\textit{From now on we fix $k \geq 1$   so  that  Theorem   \ref{Con Mixing} holds.}

We fix a right invariant metric $d$ on $\X \simeq G/ \Gamma$ induced from a right-invariant Riemannian metric on G. We denote by $B_{{G}}(\rho)$ the ball of radius $\rho$ centered at identity in ${G}$. As before, for a point $x \in {\X}$, we let ${i}(x)$ denote the injectivity radius at $x$. Clearly, the metric $d$ descends to give a right invariant metric on $\X_d$ as well. Similarly for $x \in \X_d$, we define $i_d(x)$ to be the supremum over $\rho > 0$ such that the map $B_{{G}}(\rho) \rightarrow B_{{G}}(\rho)x : g \mapsto gx$ is injective.  Also, for  $\epsilon > 0$ define
\begin{equation*}
    {\K}_{\epsilon,d} =\{ \Lambda \in {\X_d}: \|v\| > \epsilon, \text{   for all } v \in \Lambda \setminus \{0\} \},
\end{equation*}
and define $\K_\e = \pi^{-1}(\K_{\e,d})$. Note that $\K_{\e,d}$ is compact by Mahler's criterion. Since $\pi$ is a proper map, we have that $\K_\e$ is a compact subset of $\X$. For all $x \in \X$, it is clear that 
$$
i(x) \geq i_d(\pi(x)).
$$
So, using Proposition \ref{inj r hom} we may conclude the following Proposition.
\begin{prop}
\label{inj r con}
$i(x) \gg \epsilon^{d}$ for all $x \in \K_{\epsilon}$.
\end{prop}

The proof of the following Theorem is identical to the proof of Theorem 2.3 of \cite{KM2}, hence is skipped.
\begin{thm}
\label{con main lemma for mixing}
    There exists $\rho_0 > 0$ and $c,\gamma >0 $ such that for every $\rho \in (0,\rho_0),$ $f \in \C_c^{\infty}(U)$ satisfying $supp(f) \subset B_G(\rho)$, $x \in \X$ with $i(x) > 2\rho$, $\phi \in \C_c^{\infty}(\X)$, and $t \geq 0$, 

    \begin{equation}
        \int_U f(u) \phi(b_t u x ) \, du = \left( \int_U f(u) \,du \right) \left( \int_{\X} \phi \, d\mu_{\X} \right) + \mathcal{O} \left( \rho \|f\|_{L^1(U)} \|\phi \|_{Lip}  + \rho^{-c} e^{-\gamma t} \|f\|_{C^k} \|\phi\|_{\C^k} \right)
    \end{equation}
\end{thm}

As in Section \ref{Aff Mix Property}, we now embed the flow $(a_t)$ in a multi-parameter flow as follows. For $\bar{s}= (s_1, \ldots,s_{m+n}) \in \R^{m+n},$ we set  
\begin{equation*}
    a(\bar{s}) := diag(e^{s_1}, \ldots, e^{s_{m}}, e^{-s_{m+1}}, \ldots, e^{-s_{m+n}}).
\end{equation*}
Denote by $S^+$ the cone in $\R^{m+n}$ consisting of those $\bar{s}= (s_1,\ldots, s_{m+n})$ which satisfy
\begin{equation*}
    s_1,\ldots, s_{m+n} >0 \text{   and    } \sum_{i=1}^m s_i = \sum_{i= m+1}^{m+n} s_i.
\end{equation*}
For $\bar{s}= (s_1, \ldots, s_{m+n}) \in \R^{m+n}$, we set 
\begin{equation*}
    \lfloor \bar{s} \rfloor := \min(s_1, \ldots, s_{m+n}),
\end{equation*}
and, with $\bar{s}_t:= (w_1t,\ldots, w_{m+n}t)$, we see that $a_t = a(\bar{s_t})$. We have the following quantitative non-divergence estimate for unipotent flows.

\begin{thm} 
\label{con lemma 2}
    There exists $\theta= \theta(m,n) >0$ such that for every compact $L \subset \X$ and a Euclidean ball $ B \subset U$ centered at the identity, there exists $T_0 >0$ such that for every $\epsilon \in (0,1)$, $x \in L$, and $\bar{s} \in S^+$ satisfying $\lfloor \bar{s}\rfloor \geq T_0$, one has 
\begin{equation*}
    |\{ u \in B: a(\bar{s})ux \notin \K_{\epsilon}\}| \ll \epsilon^{\theta} |B|
\end{equation*}
\end{thm}
\begin{proof}
Note that 
\begin{align*}
    |\{ u \in B: a(\bar{s})ux \notin \K_{\epsilon}\}| &= |\{u \in B: \pi(a(\bar{s})ux) \notin \pi^{-1}({\K}_{\epsilon,d}) \}| \\
    &=|\{u \in B: \pi(a(\bar{s})ux) \notin {\K}_{\epsilon,d} \}| \\
    &= |\{u \in B: a(\bar{s})u\pi(x) \notin {\K}_{\epsilon,d} \}|.
\end{align*}
The theorem now  follows from the Theorem \ref{hom lemma 2}.
\end{proof}

The following theorem now follows directly from proof of corresponding theorems in \cite{BG}.  
\begin{thm}
\label{main prethm mixing con}
    There exists  $\delta > 0$ such that for every compact $\Omega \subset U$, $f \in \C_c^{\infty}(U)$ with $supp(f) \subset \Omega$, $\phi_1, \ldots, \phi_r \in \C_c^{\infty}(\X)$, $x_0 \in \X$, and $t_1, \ldots,t_r >0 $, we have 
    \begin{align*}
        \int_U f(u) \left( \prod_{i=1}^r \phi(a_{t_i} u x_0)\right) \, du &= \left(\int_U f(u) \, du  \right) \prod_{i=1}^r \left( \int_{\X} \phi(a_{t_i} u x_0) \,d\mu_{\X} \right)\\ &+ \mathcal{O}_{x_0, \Omega, r} \left(e^{-\delta D(t_1, \ldots, t_r)} \|f\|_{C^k}  \prod_{i=1}^r \| \phi_i\|_{\C^k} \right)
    \end{align*}
 where
 \begin{equation}
     D(t_1, \ldots, t_r) := \min\{t_i, |t_i - t_j|: 1 \leq i \neq j \leq r\}.
 \end{equation}
\end{thm}

\begin{cor}
\label{con aff}
    There exists  $\delta' > 0$ such that for $\phi_0 \in \C_c^{\infty}(\Y )$, $\phi_1, \ldots, \phi_r \in \C_c^{\infty}(\X)$ and $t_1, \ldots,t_r >0 $, we have 
    \begin{align*}
        \int_{\Y} \phi_0(y) \left( \prod_{i=1}^r \phi(a_{t_i} y)\right) \, d\mu_{\Y} &= \left(\int_{\Y} \phi_0 \, d\mu_{\Y}  \right) \prod_{i=1}^r \left( \int_{\X} \phi(a_{t_i} u x_0) \,d\mu_{\X} \right)\\ &+ \mathcal{O}_{x_0, \Omega, r} \left(e^{-\delta D(t_1, \ldots, t_r)} \|\phi_0\|_{C^k}  \prod_{i=1}^r \|\phi_i\|_{\C_k} \right).
    \end{align*}
\end{cor}
\begin{proof}

Let $x_0$ denote the identity coset in $\X \simeq G/ \Gamma$, which corresponds to the lattice $\Z^{m+n} + \V/N$, and recall that $$\Y= Ux_0 \simeq U/ (U \cap \Gamma).$$ Let $\Tilde{\phi}_0 \in \C^{\infty}$ denote the lift of the function $\phi_0$ to $U$, and $\chi$ the characteristic function of the subset  $$U_0 := \left\{ \begin{pmatrix}
    I_m & u \\ 0 & I_n
\end{pmatrix} : u  \in \M_{m \times n}([0,N]) \right\}.$$ Given $\epsilon > 0$, let $\chi_{\epsilon} \in \Cc(U)$ be a smooth approximation of $\chi$ with uniformly bounded support which satisfies $$\chi \leq \chi_{\epsilon} \leq 1, \ \|\chi - \chi_{\epsilon}\|_{L^1(U)} \ll \epsilon, \ \|\chi_{\epsilon}\|_{c^k} \ll \epsilon^k.$$
We observe that if $f_{\epsilon}:= \Tilde{\phi}_0 \chi_{\epsilon}$ and $f_0:= \Tilde{\phi}_0 \chi$, then $$\|f_0 - f_{\epsilon}\|_{L^1(U)} \ll \epsilon \|\phi_0\|_{C^0}$$ and $$\|f_{\epsilon}\|_{C^k} \ll \|\Tilde{\phi}_0\|_{C^k} \|\chi_{\e}\|_{C^k}, $$ which implies using equation \eqref{con covering preserves measure} that
\begin{align*}
     \int_{\Y} \phi_0(y) \left( \prod_{i=1}^r \phi(a_{t_i} y)\right) \, d\mu_{\Y} &=  \frac{1}{N^{nm}} \int_{U} f_0(u) \left( \prod_{i=1}^r \phi(a_{t_i} u x_0)\right) \, du \\
     &= \frac{1}{N^{nm}} \int_{U} f_{\e}(u) \left( \prod_{i=1}^r \phi(a_{t_i} u x_0)\right) \, du \\
     &\  + \Oo \left( \e \prod_{i=0}^{r} \|\phi_i\|_{C^0}\right), 
\end{align*}
 and $$\int_{\Y} \phi_0 \, d\mu_{\Y}= \frac{1}{N^{nm}} \int_{U} f_0(u)\, du = \frac{1}{N^{nm}} \int_{U} f_{\e}(u)\, du + \Oo(\e \|\phi_0\|_{C^0}).$$ Therefore, Theorem \ref{main prethm mixing con} implies that 
 \begin{align*}
    \int_{\Y} \phi_0(y) \left( \prod_{i=1}^r \phi(a_{t_i} y)\right) \, d\mu_{\Y} &=  \frac{1}{N^{nm}}\left(\int_U f_{\e}(u) \, du  \right) \prod_{i=1}^r \left( \int_{\X} \phi \,d\mu_{\X} \right)  \\ 
    &+ \mathcal{O}_{r} \left(\e \prod_{i=1}^r \|\phi_i\|_{C^0} + e^{-\delta D(t_1, \ldots, t_r)} \|f_{\e}\|_{C^k}  \prod_{i=1}^r \| \phi_i\|_{\C^k} \right) \\
    &= \left(\int_{\Y} \phi_0 \, d\mu_{\Y}  \right) \prod_{i=1}^r \left( \int_{\X} \phi\,d\mu_{\X} \right) \\ 
    &+ \mathcal{O}_{r} \left( \left(\e  + \e^{-k} e^{-\delta D(t_1, \ldots, t_r)} \right) \|\phi_0\|_{C^k}  \prod_{i=1}^r \| \phi_i\|_{\C^k} \right).
 \end{align*}
 The corollary (with $\delta' = \delta/(k+1)$) follows by choosing $\e= e^{-\delta D(t_1, \ldots, t_r)/(k+1)}.$ 
\end{proof}

Finally, we have the following

\begin{thm}
    Let $a= diag(a_1, \ldots, a_{m+n})$ where $a_1,\ldots, a_m >1$, $0< a_{m+1}, \ldots , a_{m+n}<1 $, and $a_1 \ldots a_{m+n}=1$. Then $a$ defines a continuous self-map of the space $\X$, which preserves $\mu_{\X}$. Then for $\phi \in \C_c^{\infty}(\X)$ and $\xi \in \R$, we have
    \begin{equation*}
        \mu_{\Y}(\{y \in \Y: \frac{1}{\sqrt{N}} \sum_{s=0}^{N-1}(\phi \circ a^s - \mu_{\Y}(\phi \circ a^s) < \xi\}) \rightarrow Norm_{\sigma_{\phi}(\xi)}
    \end{equation*}
    as $N \rightarrow \infty$, where
    \begin{equation*}
    \sigma_{\phi}^2 := \sum_{s= -\infty}^{\infty} \left( \int_{\X} (\phi \circ a^s)\phi d\mu_{\X} - \mu_{\X}(\phi)^2\right)
    \end{equation*}
\end{thm}
The proof of above theorem is the same as the proof of Theorem 3.1 of \cite{BG}, hence we omit it.

\subsection{Siegel Transforms} \label{Con Siegel} \vspace{0.3in} \hfill \\
\subsubsection{Properties of Siegel Transforms} \hfill \\

Given $f: \R^{m+n} \rightarrow \R$, we define its \textit{Siegel transform} $\hat{f}: \X \rightarrow \R$ by $$ \hat{f}(\Lambda):= \sum_{z \in \Lambda \setminus \{0\}} f(z), \text{   for } \Lambda \in \X.$$

Let us start with the definition of function $\alpha : \X  \rightarrow \R$ as 
 \begin{align}
  \alpha(x)= \alpha_{d}(\pi(x)),
\end{align}
where $\alpha_s$ is defined in \eqref{def alpha_d}. Now, we have the following results:

\begin{prop} 
\label{con alpha 1}
    If $f: \R^{m+n} \rightarrow \R$ is a bounded function with compact support, then 
    \begin{equation*}
        |\hat{f}| \ll_{supp(f)} \|f\|_{C^0} \alpha(\Lambda) \text{    for all } \Lambda \in \X
    \end{equation*}
\end{prop}
The Proposition \ref{con alpha 1} follows directly from Proposition \ref{aff alpha 1 }.

\begin{prop}
\label{con alpha 2}
    $\alpha \in L^p(\X)$ for $1 \leq p < m+n$. In particular, $\mu_{\X}(\{\alpha \geq L\}) \ll_p L^{-p}$ for all $p < m+n$. 
\end{prop}
\begin{proof}
    Since, the map $\pi: \X \rightarrow \X_d$ is measure preserving and $\alpha = \alpha_{d} \circ \pi$, we get that $\alpha \in L^p(\X)$ if and only if $\alpha_{d} \in L^p(\X_d)$. But, the latter follows from Proposition \ref{alpha 2}. The second part follows from the first one. 
\end{proof}

We now recall an analogue of Rogers's formula for the second moment of the Siegel transform in the congruence setting. This was established by the second named author, Kelmer and Yu, in \cite{GKY}, and has subsequently already found use in proving central limit theorems in lattice point counting \cite{AGH}.

\begin{prop} [\cite{GKY}  Thm 3.2]
\label{con S R}
Let $m>1$. Let $f: \R^{m+n} \rightarrow \R$ is a bounded, compactly supported function, then 

(a)\begin{equation}
\label{Se2}
    \int_{\X} \hat{f}(\Lambda) \, d\mu_{\X}(\Lambda) = \int_{\R^{m+n}}f
\end{equation}

(b) \begin{equation}
 \label{Rog2}
    \int_{\X} |\hat{f}(\Lambda)|^2 \, d\mu_{\X}(\Lambda)  =  \left(\int_{\R^{m+n}}f \right)^2 + \frac{1}{\zeta_N(m+n)} \sum_{\substack{k_1 \geq 1 \\ gcd(k_1,q)=1 }} \sum_{\substack{k_2 \in \Z \setminus \{0\} \\ k_2 = k_1 (mod \ {q})}} \int_{\R^{m+n}} f(k_1 x) f(k_2 x) \, dx.
\end{equation}
\end{prop}

\subsubsection{Non-divergence estimates of Siegel Transforms} \hfill \\

We retain the notation from Section \ref{Aff Intro}. Given 
$$
0<w_1,\ldots, w_m<n \text{ and }  w_1+\ldots+w_m=n,
$$
we denote by $a$ the self-map on $\X$ induced by
\begin{equation}
a={diag}(e^{w_1},\ldots,e^{w_m},e^{-1},\ldots,e^{-1}).
\end{equation}
Our goal in this subsection is to analyze the escape of mass for the submanifolds $a^s\Y$ and 
bound the Siegel transforms $\hat f(a^s y)$ for $y\in \Y$. We have the following propositions.

\begin{prop}
\label{Siegel COn 1}
    There exists $\kappa>0$ such that for every $L \geq 1$ and $s \geq \kappa \log L$,
    \begin{equation*}
        \mu_{\Y}(\{y \in \Y : \alpha(a^sy) \geq L\}) \ll_p L^{-p} \text{    for all } p< m+n.
    \end{equation*}
\end{prop}
\begin{proof}
    It is clear that
    \begin{align*}
       \mu_{\Y}(\{y \in \Y : \alpha(a^sy) \geq L\}) &= \mu_{\Y}(\{y \in \Y : \alpha_d(\pi(a^sy)) \geq L\}) \\
        &= \mu_{\Y}(\{y \in \Y : \alpha_{d}(a^s \pi(y)) \geq L\}) \\
        &=\mu_{\Y}(\pi^{-1}(\{z \in \Y_d : \alpha_{d}(a^s z) \geq L\})) \\
        &= \mu_{\Y_d}(\{z \in \Y_d : \alpha_{d}(a^s z) \geq L\}). 
    \end{align*}
   Now the proof follows from Proposition \ref{Siegel Hom 1}.
\end{proof}

\begin{prop} 
\label{siegel main con}
    Let $f$ be a bounded measurable function on $\R^{m+n}$ with compact support contained in the open set $\{(x_{m+1}, \ldots, x_{m+n}) \neq 0\}$. Then for $i=1,2$, we have
    \begin{equation}
       \underset{s \geq 0}{\sup} \| \hat{f} \circ a^s\|_{L^i(\Y)} < \infty.
    \end{equation}
\end{prop}
\begin{proof}
    We note that there exists $0< \upsilon_1 < \upsilon_2 $ and $\vartheta >0$ such that the support of $f$ is contained in the set $$ A= \{(\bar{x},\bar{y}) \in \R^{m+n} : \upsilon_1 \leq \|N \bar{y}\| \leq \upsilon_2, |Nx_i| \leq \vartheta\|N \bar{y}\|^{-w_i}, i=1,\ldots, n\},$$
    and without loss of generality we may assume that $f$ is the characteristic function of above set. 
    Clearly for $u \in V$, we have

    \begin{align}
        \hat{f}(a^s\Lambda_u) = \sum_{(\bar{p},\bar{q}) \in \Z^{m+n}} f\left( \frac{Np_1 + v_1+ \langle \bar{u}_1, N\bar{q} + v'' \rangle }{N} ,\ldots, \frac{Np_m + v_m+ \langle \bar{u}_m, N\bar{q} + v'' \rangle }{N} , \frac{N\bar{q} +v'' }{N}  \right).
    \end{align}
    If we  denote by $\chi'_M$ to be characteristic function of interval $[-M,M]$, then we get
    \begin{align}
        \hat{f}(a^s\Lambda_u) &= \sum_{  \upsilon_1 e^s \leq \|N\bar{q} + v''\| \leq \upsilon_2} \sum_{\bar{p} \in \Z^m}  \prod_{i=1}^m \chi'_{\frac{\vartheta}{\|N\bar{q} + v''\|^{w_i}}} \left(Np_i + v_i+ \langle \bar{u}_i, N\bar{q} + v'' \rangle   \right)  \nonumber \\
        \label{sum2}
        &= \sum_{  \upsilon_1 e^s \leq \|N\bar{q} + v''\| \leq \upsilon_2}  \prod_{i=1}^m  \left( \sum_{p_i \in \Z}  \chi'_{\frac{\vartheta}{\|N\bar{q} + v''\|}} \left(Np_i + v_i+ \langle \bar{u}_i, N\bar{q} + v'' \rangle \right)  \right) \\
         &\leq \sum_{  \upsilon_1 e^s \leq \|\bar{q}\| \leq \upsilon_2}  \prod_{i=1}^m  \left( \sum_{p_i \in \Z}  \chi'_{\frac{\vartheta}{\|\bar{q}\|}} \left(p_i + \langle \bar{u}_i, \bar{q}  \rangle \right)  \right). \nonumber
    \end{align}

    Thus by \eqref{con covering preserves measure}, we have  
    \begin{align}
        \int_{\Y} |\hat{f} \circ a^s| \, d\mu_{\Y} &= \frac{1}{N^{nm}} \left( \int_{u \in \M_{m \times n}([0,N])} \hat{f}(a^s\Lambda_u) \, du \right) \nonumber \\
        &\leq \frac{1}{N^{nm}} \left( \sum_{  \upsilon_1 e^s \leq \|\bar{q}\| \leq \upsilon_2}  \prod_{i=1}^m \left(  \sum_{p_i \in \Z}  \int_{[0,N]^n} \chi'_{\frac{\vartheta}{\|\bar{q} \|}} \left(p_i + \langle \bar{u}_i, \bar{q} \rangle  \right)  d\bar{u_i} \right) \right) \nonumber \\
        &= \sum_{  \upsilon_1 e^s \leq \|\bar{q}\| \leq \upsilon_2}  \prod_{i=1}^m \left(  \sum_{p_i \in \Z}  \int_{[0,1]^n} \chi'_{\frac{\vartheta}{\|\bar{q} \|}} \left(p_i + \langle \bar{u}_i, \bar{q} \rangle  \right)  d\bar{u_i} \right), \label{con 1 1}
    \end{align}
    and
    \begin{align}
        \int_{\Y} |\hat{f} \circ a^s|^2 \, d\mu_{\Y} &= \frac{1}{N^{nm}} \left( \int_{u \in \M_{m \times n}([0,N])} \hat{f}(a^s\Lambda_u)^2 \, du \right)  \nonumber\\
        &\leq \frac{1}{N^{nm}} \left( \sum_{  \upsilon_1 e^s \leq \|\bar{q}\|, \|\bar{l}\| \leq \upsilon_2}  \prod_{i=1}^m \left(  \sum_{p_i, r_i \in \Z}  \int_{[0,N]^n} \chi'_{\frac{\vartheta}{\|\bar{q} \|}} \left(p_i + \langle \bar{u}_i, \bar{q} \rangle  \right) \chi'_{\frac{\vartheta}{\|\bar{l} \|}} \left(r_i + \langle \bar{u}_i, \bar{l} \rangle  \right)  d\bar{u_i} \right) \right) \nonumber \\
        &= \sum_{  \upsilon_1 e^s \leq \|\bar{q}\|, \|\bar{l}\| \leq \upsilon_2}  \prod_{i=1}^m \left(  \sum_{p_i, r_i \in \Z}  \int_{[0,1]^n} \chi'_{\frac{\vartheta}{\|\bar{q} \|}} \left(p_i + \langle \bar{u}_i, \bar{q} \rangle  \right) \chi'_{\frac{\vartheta}{\|\bar{l} \|}} \left(r_i + \langle \bar{u}_i, \bar{l} \rangle  \right)  d\bar{u_i} \right). \label{con 1 2}
    \end{align}

   Note that the expression \eqref{con 1 1} is the same as equation (4.3) of \cite{BG}. The boundedness of the latter expression, independent of $s$, was proved in Proposition 4.6 of \cite{BG}. Similarly, the expression \eqref{con 1 2} is the same as the equation obtained in Proposition 4.8 of \cite{BG}, and the boundedness of the  latter expression, independent of $s$, was the content of the same Proposition.
\end{proof}

\subsubsection{Truncated Siegel Transforms} \label{Con Truncated Siegel}\hfill \\

As in Section \ref{Aff Truncated Siegel}, we define for a bounded function $f: \R^{m+n} \rightarrow \R$ with compact support, the \textit{truncated Siegel transform} of $f$ as $$\hat{f}^{(L)} := \hat{f}. \eta_L,$$ where $\eta_L$ is defined below in Lemma \ref{con trun 1}. We also record some basic properties of this transform in Lemma \ref{con trun 2}.
\begin{lem}
\label{con trun 1}
    For every $c>1$, there exists a family $(\eta_L)$ in $\C_c^{\infty}(\X)$ satisying
    \begin{equation*}
        0\leq \eta_L \leq 1,\ \ \eta_L= 1 \text{ on } \{\alpha \leq c^{-1}L\}, \ \ \eta_L =0 \text{ on } \{ \alpha > cL\}, \ \ \|\eta_L\|_{C^k} \ll 1.
    \end{equation*}
\end{lem}
\begin{lem}
\label{con trun 2}
    For $f \in \C_c^{\infty}(\R^{m+n})$, the truncated Siegel transform $\hat{f}^{(L)}:= \hat{f}. \eta_L$ is in $\C-C^{\infty}(\X)$, and it satisfies
    \begin{align*}
      \|\hat{f}^{(L)}\|_{L^p(\X)} \leq \|\hat{f}\|_{L^p(\X)} &\ll_{supp(f),p} \|f\|_{C^0} \text{    for all } p<m+n  \\
      \|\hat{f}^{(L)}\|_{C^0} &\ll_{supp(f)} L \|f\|_{C^0}  \\
      \|\hat{f}^{(L)}\|_{C^k} &\ll_{supp(f)} L \|f\|_{C^k}  \\
      \|\hat{f}-\hat{f}^{(L)}\|_{L^1(\X)} &\ll_{supp(f),\tau} L^{-\tau}\|f\|_{C^0} \text{    for all } \tau<m+n-1 \\
      \|\hat{f}-\hat{f}^{(L)}\|_{L^2(\X)} &\ll_{supp(f),\tau} L^{-(\tau-1)/2}\|f\|_{C^0} \text{    for all } \tau<m+n-1. 
    \end{align*}
    Moreover, the implied constants are uniform when $supp(f)$ is contained in a fixed compact set.
\end{lem}
The proofs of Lemmas \ref{con trun 1}, \ref{con trun 2} are exact replicas of the proofs of Lemmas 4.11 and 4.12 of \cite{BG} respectively, hence are skipped.

\subsection{CLT for Smooth Siegel Transform}
 
\begin{thm}
    Assume $f \in \C_c^{\infty}(\R^{m+n})$ satisfies $f \geq 0$ and $supp(f) \subset \{(x_{m+1}, \ldots, x_{m+n}) \neq 0\}$. Assume  $m \geq 2$. Then, consider the sequence of averages
    \begin{equation*}
        F_M(y):= \frac{1}{\sqrt{M}} \sum_{s=0}^{M-1}(\hat{f}(a^s)-  \mu_{\Y}(\hat{f} \circ a^s)) \text{     with } y \in \Y
    \end{equation*}
    where $a={diag}(e^{w_1},\ldots,e^{w_m},e^{-1},\ldots,e^{-1}) $ . Then the variance
    \begin{equation*}
        \sigma_f^2 := \frac{1}{\zeta_N(m+n)} \sum_{s= -\infty}^{\infty}   \left(\sum_{\substack{k_1 \geq 1 \\ gcd(k_1,q)=1 }} \sum_{\substack{k_2 \in \Z \setminus \{0\} \\ k_2 = k_1 (mod \ {q})}} \int_{\R^{m+n}} \left( (f \circ a^s)(k_1 x) f(k_2 x) \right) \, dx \right)
    \end{equation*}
    is finite, where $\zeta_N(r):= \sum_{\substack{ k \geq 1 \\ \gcd(k,N)=1}} k^{-r}$. Also, for every $\xi \in \R$,
    \begin{equation*}
        \mu_{\Y}(\{ y \in \Y : F_M(y) < \xi\}) \rightarrow Norm_{\sigma_f}(\xi)
    \end{equation*}
    as $M\rightarrow \infty$.
\end{thm}
\begin{proof}
    All the steps given in the proof of [Thm 5.1, \cite{BG}] apply here directly to give the result here. The only change required is after [eq (5.36) \cite{BG}], where we are computing the value of variance. More precisely, following the same steps we get
    \begin{equation*}
        \|F_M^{(L)}\|_{L^2(\Y)} \rightarrow \Theta_{\infty}(0) + 2 \sum_{s=1}^{\infty} \Theta_{\infty}(s) = \sum_{s= -\infty}^{\infty} \Theta_{\infty}(s)
    \end{equation*}
    as $M \rightarrow \infty$, where
    \begin{align}
        \Theta_{\infty}(s) &:= \int_{\X}(\hat{f} \circ a^s) \hat{f}\, d\mu_{\X} -\mu_{\X}(\hat{f})^2.
    \end{align}

Using Proposition \ref{con S R}, we get that 
\begin{align*}
    \Theta_\infty (s) &= \int_{\X} (\hat f\circ a^s)\hat f\, d\mu_\X -\mu_\X(\hat f)^2 \\
    &=  \frac{1}{2} \left(\int_{\X}(\hat{f} \circ a^s + \hat{f})^2\, d\mu_{\X} - 2\int_{\X}( \hat{f})^2\, d\mu_{\X}  -2\mu_\X(\hat f)^2 \right) \\
    &= \frac{1}{2} \left(\left( \int_{\R^d}(f \circ a^s + f) \right)^2 - 2  \left( \int_{\R^d} f \right)^2  -2 \left( \int_{\R^d} f \right)^2  \right)\\
    &+\frac{1}{2\zeta_N(m+n)} \sum_{\substack{k_1 \geq 1 \\ gcd(k_1,q)=1 }} \sum_{\substack{k_2 \in \Z \setminus \{0\} \\ k_2 = k_1 (mod \ {q})}} \int_{\R^{m+n}} ( (f \circ a^s +f)(k_1 x) (f \circ a^s +f)(k_2 x)\\ &-2 f(k_1 x) f(k_2 x)) \, dx  \\
    &= \frac{1}{\zeta_N(m+n)} \left(\sum_{\substack{k_1 \geq 1 \\ gcd(k_1,q)=1 }} \sum_{\substack{k_2 \in \Z \setminus \{0\} \\ k_2 = k_1 (mod \ {q})}} \int_{\R^{m+n}} \left( (f \circ a^s)(k_1 x) f(k_2 x) \right) \, dx \right).
\end{align*} 
Hence, we get that
    \begin{align}
    \label{var2}
     \sum_{s= -\infty}^{\infty} \Theta_{\infty}(s) &= \frac{1}{\zeta_N(m+n)} \sum_{s= -\infty}^{\infty}   \left(\sum_{\substack{k_1 \geq 1 \\ gcd(k_1,q)=1 }} \sum_{\substack{k_2 \in \Z \setminus \{0\} \\ k_2 = k_1 (mod \ {q})}} \int_{\R^{m+n}} \left( (f \circ a^s)(k_1 x) f(k_2 x) \right) \, dx \right) \\
     \label{var2'}
     & \leq \frac{1}{\zeta_N(m+n)} \sum_{s= -\infty}^{\infty}   \left(\sum_{k_1 \geq 1} \sum_{k_2 \in \Z \setminus \{0\} } \int_{\R^{m+n}} \left( (f \circ a^s)(k_1 x) f(k_2 x) \right) \, dx \right) .  
    \end{align}
    The finiteness of the sum in \eqref{var2'} is shown in [\cite{BG}, Thm 5.1]. Hence \eqref{var2} is also finite. Then, resuming the steps again from [Section 5.3 in \cite{BG}], we complete the proof.
\end{proof}

\subsection{CLT for counting functions and the proof of Theorem \ref{CLT3}} \hfill \\

Let $$\Omega_T := \{(\bar{x},\bar{y}) \in \R^{m+n}:1\leq \|N \bar{y}\| < T, |Nx_i|< \vartheta_i \|N\bar{y}\|^{-w_i}, i=1,\ldots,m\}.$$ Then, it is clear that $$\Delta_T(u,v)= |\Lambda_{u,v} \cap \Omega_T| + \mathcal{O}(1) \text{,   for } T>0.  $$ With $a= {diag}(e^{w_1},\ldots,e^{w_m},e^{-1},\ldots,e^{-1})$,  we note that for any integer $M \geq 1$, $$\Omega_{e^M} =  \sqcup_{s=0}^{M-1}a^{-s} \Omega_e, $$ and thus $$|\Delta_{u,v} \cap \Omega_{e^M}| = \sum_{s=0}^{M-1} \hat{\chi}(a^s \Lambda_{u,v}),$$ where $\chi$ denotes the characteristic function of the set $\Omega_e$. Define $$F_M:= \frac{1}{\sqrt{M}} \sum_{s=0}^{N-1} \left( \hat{\chi} \circ a^s - \mu_{\Y}(\hat{\chi} \circ a^s) \right).$$

\begin{thm}
\label{CLT3'}
    If $m \geq 2$, then for every $\xi \in \R$, 
    \begin{equation*}
        \mu_{\Y}(\{ y \in \Y : F_M(y) < \xi\}) \rightarrow Norm_{\sigma}(\xi)
    \end{equation*}
    as $M\rightarrow \infty$, where 
    \begin{equation*}
    \sigma^2= 2^m \vartheta_1\vartheta_2 \ldots \vartheta_m \omega_n \text{ with } \omega_n := \int_{S^{n-1}}  d\bar{z}.
\end{equation*}
\end{thm}
\begin{proof}
    As before, the same steps given in proof of [Thm 6.1, \cite{BG}] work here. The only changes needed include reproving [Prop 6.2 \cite{BG}] in our case, and computation of the variance $\sigma^2$. Hence we show only these steps.
    \begin{prop}
        \label{Con approx 1}
        For every $s \geq 0$, $$\int_{\Y} \left| \hat{f}_{\e} \circ a^s - \hat{\chi} \circ a^s\right| \, d\mu_{\Y} \ll \e + e^{-s}, $$
        where $ f_{\epsilon} \in \Cc(\R^{m+n})$ is a family of non-negative smooth functions approximating $\chi$, whose supports are contained in an $\epsilon$-neighbourhood of the set $\Omega_e$, and $$\chi \leq f_{\epsilon} \leq 1,\  \|f_{\e} - \chi\|_{L^1(\R^{m+n})} \ll \e, \ \| f_{\e}- \chi\|_{L^2(\R^{m+n})} \ll \e^{1/2}, \ \|f_{\e}\|_{C^k} \ll \e^{-k}.$$
    \end{prop}
    \begin{proof}
        We observe that there exists $\vartheta_{i}(\epsilon)> \vartheta_i$ such that $\vartheta_i(\epsilon)= \vartheta_i + \mathcal{O}(\epsilon)$ and $f_{\epsilon} \leq \chi_{\epsilon}$, where $\chi_{\epsilon}$ denotes the characteristic function of the set 
        \begin{equation}
            \{(\bar{x},\bar{y}) \in \R^{m+n}: 1- \epsilon \leq \|N\bar{y}\| \leq e +\epsilon, |Nx_i| < \vartheta_{i}(\epsilon) \|N\bar{y}\|^{-w_i} \text{  for } i= 1,\ldots, m\}.
        \end{equation}

        Then, it follows that
        \begin{equation*}
            |\hat{f_{\epsilon}}(a^s \Lambda) - \hat{\chi}(a^s\Lambda)| = \sum_{v \in \Lambda \setminus \{0\}} (f_{\epsilon}(a^sv)- \chi(a^sv)) \leq \sum_{v \in \Lambda \setminus \{0\}} (\chi_{\epsilon}(a^sv)- \chi(a^sv)) .
        \end{equation*}
        It is clear that $\chi_{\epsilon}- \xi$ is bounded by the sum $\chi_{1,\epsilon} + \chi_{2,\epsilon} + \chi_{3,\epsilon}$ of the characteristic functions of the sets
        \begin{eqnarray*}
        &\{& (\bar{x},\bar{y}) \in \R^{m+n} : 1-\epsilon \leq \|N\bar{y}\| \leq 1,  |Nx_i| < \vartheta_{i}(\epsilon) \|N\bar{y}\|^{-w_i} \text{  for } i= 1,\ldots, m\} \\
        &\{&  (\bar{x},\bar{y}) \in \R^{m+n} : e \leq \|N\bar{y}\| \leq  e+ \epsilon,  |Nx_i| < \vartheta_{i}(\epsilon) \|N\bar{y}\|^{-w_i} \text{  for } i= 1,\ldots, m\} \\
        &\{&  (\bar{x},\bar{y}) \in \R^{m+n} : 1 \leq \|N\bar{y}\| \leq  e,  |Nx_i| < \vartheta_{i}(\epsilon) \|N\bar{y}\|^{-w_i} \text{  for } i= 1,\ldots, m, |x_j|\geq \vartheta_j \|\bar{y}\|^{-w_j} \text{  for some } j\}
        \end{eqnarray*}
        respectively. In particular, we obtain that
        \begin{equation*}
            \hat{f}_{\epsilon}(a^s \Lambda) - \hat{\chi}(a^s\Lambda) \leq \hat{\chi}_{1,\epsilon}(a^s \Lambda) + \hat{\chi}_{2,\epsilon}(a^s \Lambda) + \hat{\chi}_{3,\epsilon}(a^s \Lambda).
        \end{equation*}
        Hence, it remains to show that for $j=1,2,3$, 
        \begin{equation*}
            \int_{\Y} (\hat{\chi}_{j, \epsilon} \circ a^s) \, d\mu_{\Y} \ll \epsilon + e^{-s}.
        \end{equation*}
        As in \eqref{sum2}, we have
        \begin{eqnarray}
            \int_{\Y} \hat{\chi}_{1, \epsilon} \circ a^s  \, d\mu_{\Y} = \frac{1}{N^{nm}} \underset{(1- \epsilon)e^s \leq \|N \bar{q} + v''\| \leq e^s}{ \sum} \prod_{i=1}^{m} \left(  \underset{\bar{p_i} \in \Z}{ \sum} \int_{[0,N]^n}  \chi_{\vartheta_i(\epsilon) \|N\bar{y} + v'' \|^{-w_i}}'(Np_i+ v_i + \langle \bar{u}_i, N\bar{q} + v'' \rangle )   \, d\bar{u_i} \right). \nonumber 
        \end{eqnarray}
        We observe that 
        \begin{equation*}
            \int_{[0,N]^n}  \chi_{\frac{\vartheta_i(\epsilon)}{ \|N\bar{y} + v''\|^{w_i}}}'(Np_i+ v_i + \langle \bar{u}_i, N\bar{q} + v'' \rangle ) \, d\bar{u}_i \ll (\max_k |Nq_k + v_{m+k}|)^{-1} \|N \bar{q} + v''\|^{-w_i} \ll \|N \bar{q} + v''\|^{-1-w_i},
        \end{equation*}
        and moreover this integral is non-zero only when $|p_k|= \mathcal{O}(\|N\bar{q} + v''\|)$. Hence, 
        \begin{equation*}
            \underset{\bar{p_i} \in \Z}{ \sum} \int_{[0,N]^n}  \chi_{\vartheta_i(\epsilon) \|N\bar{y} + v''\|^{-w_i}}'(Np_i+ v_i \langle \bar{u}_i, N\bar{q} + v'' \rangle ) \, dv_i  \, d\bar{u_i} \ll \|N \bar{q} + v''\|^{-w_i}
        \end{equation*}
        and
        \begin{eqnarray*}
            \int_{\Y} (\hat{\chi}_{1, \epsilon} \circ a^s) \, d\mu_{\Y} &\ll& \underset{(1-\epsilon)e^s \leq \|N\bar{q} + v''\| \leq e^s}{ \sum} \prod_{i=1}^{m} \|N\bar{q} + v''\|^{-w_i}  \\
            &\ll& e^{-ns} |\{ \bar{q} \in \Z^n: (1-\epsilon)e^s \leq \|N \bar{q} + v''\| \leq e^s \}| \\
            &\ll& e^{-ns} |\{ \bar{q} \in \Z^n: (1-\epsilon)e^s \leq \| \bar{q} \| \leq e^s \}|.
        \end{eqnarray*}
        The number of integral points in the region $\{ (1-\epsilon)e^s \leq \|\bar{q}\| \leq e^s \}$ can be estimated in terms of its volume. Namely, there exists $r>0$ (depending only on the norm) such that 
        \begin{equation*}
           |\{ \bar{q} \in \Z^n: (1-\epsilon)e^s \leq \|\bar{q}\| \leq e^s \}| \ll  |\{ \bar{y} \in \R^n: (1-\epsilon)e^s -r \leq \|\bar{y}\| \leq e^s +r \}|.
        \end{equation*}
        Hence, 
        \begin{eqnarray*}
            \int_{\Y} (\hat{\chi}_{1, \epsilon} \circ a^s) \, d\mu_{\Y} &\ll& e^{-ns} ((e^s +r)^n- ((1-\epsilon)e^s-r)^n) \\
            &=& (1 +re^{-s})^n- ((1-\epsilon)-re^{-s})^n \\
            &\ll& \epsilon + e^{-s}.
        \end{eqnarray*}
        The integral $\hat{\chi}_{2,\epsilon} \circ a^s$ can be estimated similarly. 

        The integral over $\hat{\chi}_{3,\epsilon} \circ a^s$  calculated similarly and written as a sum of the product of the integrals
        \begin{eqnarray*}
            \int_{[0,N]^n} & \left( \chi_{\frac{\vartheta_j(\epsilon)}{\|N \bar{q}+  v''\|^{w_j}}}'(Np_j + v_j + \langle\bar{u}_j , N \bar{q} + v'' \rangle) - \chi_{\frac{\vartheta_j}{\|\|^{w_j}}}'(Np_j + v_j + \langle\bar{u}_j , N \bar{q} + v'' \rangle)\right) \, d\bar{u}_j \\
            & \ll 2 (\max_k |Nq_k + v_{m+k}|)^{-1}(\vartheta_j(\epsilon)-\vartheta_j) \|N\bar{q}+ v''\|^{-w_j} \ll \epsilon\|N \bar{q} + v''\|^{-1-w_j}
        \end{eqnarray*}
        and the integrals 
        \begin{equation*}
             \int_{[0,N]^n}   \chi_{\frac{\vartheta_i(\epsilon)}{\|N \bar{q}+  v''\|^{w_j}}}'(Np_i + v_i + \langle\bar{u}_i , N \bar{q} + v'' \rangle)  \, d\bar{u}_i \ll 2 (\max_k |Nq_k + v_{m+k}|)^{-1}\vartheta_i(\epsilon) \|N\bar{q}+ v''\|^{-w_i} \ll \|N \bar{q} + v''\|^{-1-w_i}
        \end{equation*}
        with $i \neq j$. We observe that these integrals are non zero only when $|p_j|= \mathcal{O}(\|N \bar{q}+  v''\|)$ and $|p_i|= \mathcal{O}(\|N \bar{q}+  v''\|)$. Hence, we conclude that
        \begin{eqnarray*}
           \int_{\Y} (\hat{\chi}_{3, \epsilon} \circ a^s) \, d\mu_{\Y} \ll   \underset{e^s \leq \|N\bar{q} + v''\| \leq e^{s+1}}{ \sum} \epsilon \prod_{i=1}^{m}  \|N\bar{q} + v''\|^{-w_i} \ll \epsilon \left(\underset{e^s \leq \|\bar{q}\| \leq e^{s+1}}{ \sum}  \|\bar{q}\|^{-n} \right) \ll \epsilon.
        \end{eqnarray*}
        which completes the proof of the proposition.
     \end{proof}

    Now, we compute the variance. Similar to the proof of Theorem 6.1 of \cite{BG}, we get that
        \begin{equation}
            \label{con eq: main variance}
            \left\|\Tilde{F}^{(\e,L)}_M\right\|_{L^2(\Y)}^2 =
            \Theta^{(\e)}_\infty(0)+2\sum_{s=1}^{K-1}\Theta^{(\e)}_\infty (s)+o(1),
        \end{equation}
    where 
    $$
    \Theta^{(\e)}_\infty(s):=\int_{\X} (\hat f_\e\circ a^s)\hat f_\e\, d\mu_\X -\mu_\X(\hat f_\e)^2.
    $$
    Using Proposition \ref{con S R}, we get that 
\begin{align*}
    \Theta^{(\e)}_\infty (s) &= \int_{\X} (\hat f_\e\circ a^s)\hat f_\e\, d\mu_\X -\mu_\X(\hat f_\e)^2 \\
    &=  \frac{1}{2} \left(\int_{\X}(\hat{f_\e} \circ a^s + \hat{f}_\e)^2\, d\mu_{\X} - 2\int_{\X}( \hat{f}_\e)^2\, d\mu_{\X}  -2\mu_\X(\hat f_\e)^2 \right) \\
    &= \frac{1}{2} \left(\left( \int_{\R^d}(f_\e \circ a^s + f_\e) \right)^2 - 2  \left( \int_{\R^d} f_\e \right)^2  -2 \left( \int_{\R^d} f_\e \right)^2  \right)\\
    &+\frac{1}{2\zeta_N(m+n)} \sum_{\substack{k_1 \geq 1 \\ gcd(k_1,q)=1 }} \sum_{\substack{k_2 \in \Z \setminus \{0\} \\ k_2 = k_1 (mod \ {q})}} \int_{\R^{m+n}} \left( (f_\e \circ a^s +f_\e)(k_1 x) (f_\e \circ a^s +f_\e)(k_2 x) \right. \\
    &\left. - 2 f_\e(k_1 x) f_\e(k_2 x) \right) \, dx  \\
    &= \frac{1}{\zeta_N(m+n)} \left(\sum_{\substack{k_1 \geq 1 \\ gcd(k_1,q)=1 }} \sum_{\substack{k_2 \in \Z \setminus \{0\} \\ k_2 = k_1 (mod \ {q})}} \int_{\R^{m+n}} \left( (f_\e \circ a^s)(k_1 x) f_\e(k_2 x) \right) \, dx \right).
\end{align*} \\

Now, for $s \in \Z$ define
    \begin{align}
        \Theta_{\infty}(s) &:= \frac{1}{\zeta_N(m+n)} \left(\sum_{\substack{k_1 \geq 1 \\ gcd(k_1,q)=1 }} \sum_{\substack{k_2 \in \Z \setminus \{0\} \\ k_2 = k_1 (mod \ {q})}} \int_{\R^{m+n}} \left( (\chi \circ a^s)(k_1 x) \chi(k_2 x) \right) \, dx \right).
    \end{align}
It is clear that since $f_\e \geq \chi$ so 
\begin{align}
    &|\Theta^{(\e)}_\infty (s) - \Theta_{\infty}(s)| = \Theta^{(\e)}_\infty (s) - \Theta_{\infty}(s) \nonumber \\
    &= \frac{1}{\zeta_N(m+n)} \left(\sum_{\substack{k_1 \geq 1 \\ gcd(k_1,q)=1 }} \sum_{\substack{k_2 \in \Z \setminus \{0\} \\ k_2 = k_1 (mod \ {q})}} \int_{\R^{m+n}} \left( (f_\e \circ a^s)(k_1 x) f_\e(k_2 x)-  (\chi \circ a^s)(k_1 x) \chi(k_2 x)  \right) \, dx \right) \nonumber\\
    &\leq \frac{1}{\zeta_N(m+n)} \left(\sum_{\substack{k_1 \geq 1 }} \sum_{\substack{k_2 \in \Z \setminus \{0\} }} \int_{\R^{m+n}} \left( (f_\e \circ a^s)(k_1 x) f_\e(k_2 x)-  (\chi \circ a^s)(k_1 x) \chi(k_2 x)  \right) \, dx \right). \label{con 2 1}
\end{align}
The expression in \eqref{con 2 1} is the same as the left hand side of equation (6.17) of \cite{BG}. Hence by equation (6.17) of \cite{BG}, we get that the expression in \eqref{con 2 1} is $\ll \e^{1/2}e^{-(m+n-2)s/2}$. Thus, we have
\begin{equation}
\label{con 2 2}
|\Theta^{(\e)}_\infty(s)-\Theta_\infty(s)|\ll \e^{1/2}e^{-(m+n-2)s/2}.
\end{equation}
Provided that 
$\e=\e(M)\to 0,$
the estimate \eqref{con 2 2} implies that
$$
\left\| \Tilde{F}^{(\e,L)}_M\right\|_{L^2(\Y)}^2 =
\Theta_\infty(0)+2\sum_{s=1}^{K-1}\Theta_\infty (s)+o(1).
$$
Hence, 
$$
\left\|\Tilde{F}^{(\e,L)}_M\right\|_{L^2(\Y)}^2 \to \sigma^2:=\Theta_\infty(0)+2\sum_{s=1}^{\infty}\Theta_\infty (s)
$$
as $M\to\infty$, provided that the latter is finite.

\vspace{0.2cm}

Finally, we compute the limit
        \begin{align*}
            \sigma^2 = \sum_{s= - \infty}^{\infty} \Theta_{\infty}(s)=  \frac{2}{\zeta_N(m+n)} \sum_{s= -\infty}^{\infty}   \left( \sum_{\substack{p \geq 1 \\ gcd(p,N)=1 }} \sum_{\substack{q \geq 1 \\ q = p (mod \ {N})}} \int_{\R^{m+n}} \chi(p a^s x) \chi(q x)  \, dx \right).
        \end{align*}
We note that the sum  $$ \Xi:= \sum_{s= -\infty}^{\infty} \chi \circ a^s$$ is equal to the characteristic function of the set $$\{(\bar{x}, \bar{y}) \in \R^{m+n}: \|\bar{y}\| >0 , |Nx_i| < \vartheta_i\|N\bar{y}\|^{-w_i}, i=1, \ldots,m\},$$ and
        \begin{align*}
            &\int_{\R^{m+n}} \Xi(px) \chi(qx) \, dx = \int_{1/q < \|N\bar{y}\| < e/q} \left( \prod_{i=1}^m 2 \vartheta_i \max(Np, Nq)^{-1-w_i} \|\bar{y}\|^{-w_i}\right) d\bar{y} \\
            &= \frac{2^m}{N^{m+n}} \left( \prod_{i=1}^{m} \vartheta_i\right) \max(p, q)^{-m-n} \int_{1/Nq < \|\bar{y}\| < e/Nq} \|\bar{y}\|^{-w_i} d\bar{y} \\
            &=\frac{2^m}{N^{m+n}} \left( \prod_{i=1}^{m} \vartheta_i\right) \max(p, q)^{-m-n} \int_{S^{n-1}} \|\bar{z}\|^{-n} \left( \int_{ N^{-1}q^{-1} \|\bar{z}\|^{-1}}^{e N^{-1}q^{-1} \|\bar{z}\|^{-1}} r^{-1} \, dr  \right) d\bar{z} \\
            &= \frac{2^m}{N^{m+n}} \left( \prod_{i=1}^{m} \vartheta_i\right) \max(p, q)^{-m-n} \omega_n,
        \end{align*}
        where $\omega_n = \int_{S^{n-1}}  \, d\bar{z}$. Let $S_N=\{ i \in \Z : 0 \leq i < N, \gcd(i,N)= 1\}$. We also see that
        \begin{align*}
           \sum_{\substack{p \geq 1 \\ gcd(p,N)=1 }} \sum_{\substack{q \geq 1 \\ p = q (mod \ {N})}}  \max(p, q)^{-m-n}  &= \sum_{r \in S_N}\left( \sum_{p,q \geq 1} \max(Np+r, Nq+r)^{-m-n} \right) \\
           &= \sum_{r \in S_N}\left( \sum_{p\geq 1} (Np+r)^{-m-n} + 2 \sum_{1 \leq p <q} (Nq+r)^{-m-n} \right) \\
           &= \zeta_N(m+n) + 2 \sum_{r \in S_N} \sum_{q \geq 1} \frac{q-1}{(Nq+r)^{m+n}}
        \end{align*}
        where $\varphi$ denotes the Euler's Totient function. Thus 
        \begin{equation*}
            \sigma^2 = \sum_{s= - \infty}^{\infty} \Theta_{\infty}(s)= \frac{2^{m+1}}{N^{n+m}}\left( \prod_{i=1}^{m} \vartheta_i\right) \omega_n \left(  1 + \frac{2}{\zeta_N(m+n)} \sum_{r \in S_N} \sum_{q \geq 1} \frac{q-1}{(Nq+r)^{m+n}}\right).
        \end{equation*}
This finishes the proof of Theorem \ref{CLT3'}.
\end{proof}

\begin{proof}[\bf Proof of Theorem \ref{CLT3} ]
    The only step that require change is Lemma 6.3 of \cite{BG}. We need to compute 
    \begin{equation*}
       \sum_{s=0}^{M-1} \int_{\Y} \hat{\chi}(a^sy) \, d\mu_{\Y}(y)  = \int_{\Y} \hat{\Xi}_M (y) \, d\mu_{\Y}(y),
    \end{equation*}
    where $\Xi_M$ denotes the characteristic function of the set $$\{(\bar{x}, \bar{y}) \in \R^{m+n} : 1 \leq \|N\bar{y}\| < e^M, |Nx_i| < \vartheta_i \|N \bar{y}\|^{-w_i}, i=1, \ldots,m \}.$$ 
    Clearly, we have 
    \begin{align*}
        \int_{\Y} \hat{\Xi}_M(y) \, d\mu_{\Y}(y) &= \frac{1}{N^{nm}}\sum_{1 \leq \|N\bar{q}+ v''\| \leq e^M} \sum_{\bar{p} \in \Z^m} \prod_{i=1}^m \int_{[0,N]^n} \chi_{\frac{\vartheta_i}{\|N \bar{q} + v''\|^{w_i}}}'(Np_i + v_{m+i} + \langle \bar{u}_i, N\bar{q} + v''\rangle) \, d\bar{u}_i \\
        &= \frac{1}{N^{nm}}\sum_{1 \leq \|N\bar{q}+ v''\| \leq e^M}  \prod_{i=1}^m  \left( \sum_{p_i \in \Z}\int_{[0,N]^n} \chi_{\frac{\vartheta_i}{\|N \bar{q} + v''\|^{w_i}}}'(Np_i + v_{m+i} + \langle \bar{u}_i, N\bar{q} + v''\rangle) \, d\bar{u}_i \right).  
    \end{align*}
    We claim that 
    \begin{equation}
    \label{value}
         \sum_{p_i \in \Z}\int_{[0,N]^n} \chi_{\frac{\vartheta_i}{\|N \bar{q} + v''\|^{w_i}}}'(Np_i + v_{m+i} + \langle \bar{u}_i, N\bar{q} + v''\rangle) \, d\bar{u}_i = 2 N^{n-1} \vartheta_i \|N \bar{q} + v''\|^{-w_i}.
    \end{equation}
    To prove this, let us consider more generally a bounded measurable function $h$ on $\R$ with compact support, the function $\psi(x)= h(x_1)$ on $\R^n$, and the function $\Tilde{\psi} (x)= \sum_{p \in \Z}\psi(Np +x_1)$ on the torus $\R^n/ N\Z^n$. We suppose without loss of generality that $Nq_1+ v_{m+1} \neq 0$ and consider a non-degenerate linear map $$ S: \R^n \rightarrow \R^n : \bar{u} \mapsto (\langle \bar{u}, N \bar{q}+ v''\rangle + v_{m+1}, u_2, \ldots, u_n )$$ which induced a linear epimorphism of the torus $\R^n/ N\Z^n$. Using that $S$ preserves the Lebesgue probability measure $\mu$ on $\R^n/ N\Z^n$, we deduce that
    \begin{align*}
        \sum_{p \in \Z}\int_{[0,N]^n} h(Np + v_{m+1} + \langle \bar{u}, N\bar{q} + v''\rangle) \, d\bar{u} &= \int_{\R^n/ N\Z^n} \Tilde{\psi}(Sx) \, d\bar{u} \\
        &= \int_{\R^n/ N\Z^n} \Tilde{\psi}(x) \, d\bar{u} = N^{n-1} \int_{\R} h(u_1) \, du_1,
    \end{align*}
    which yields \eqref{value}.
    In turns, \eqref{value} implies that
    \begin{equation*}
        \int_{\Y} \hat{\Xi}_M \, d\mu_{\Y} = \frac{2^m}{N^{m+n}} \left( \prod_{i=1}^m \vartheta_i \right) \sum_{1/N \leq \|\bar{q} + v''/N\| < e^M/N} \|\bar{q} + v''/N\|^{-n}.
    \end{equation*}
    Using that $\|\bar{y_2}\|^{-n}= \|y_1\|^{-n} + \mathcal{O}(\|\bar{y}\|^{-n-1})$ when $\|\bar{y_1} - \bar{y_2}\| \ll 1$, we deduce that $$\underset{1/N \leq \|\bar{q} + v''/N\| \leq e^M/N}{ \sum} \|\bar{q} + v''/N\|^{-n} = \int_{1/N \leq \|\bar{y}\| \leq e^M/N}   \|\bar{y} \|^{-n} \, d\bar{y} +\mathcal{O}(1),$$ and expressing the integral in polar co-ordinates, we obtain $$\int_{1/N \leq \|\bar{y} \| \leq e^M/N}   \|\bar{y} \|^{-n} \, d\bar{y} = \int_{S^{n-1}} \int_{\|z\|^{-1}/N}^{\|z\|^{-1}e^M/N}  \|r \bar{z}\|^{-n} r^{n-1} \, dr \, d\bar{z}= \omega_nM +\mathcal{O}(1).$$ This proves the corresponding lemma.
    Rest steps are same.
\end{proof}

\bibliography{Biblio}

\begin{bibdiv}
\begin{biblist}

\bib{AGH}{article}{
      author={Alam, Mahbub},
      author={Ghosh, Anish},
      author={Han, Jiyoung},
       title={Higher moment formulae and limiting distributions of lattice
  points},
     journal={arxiv preprint},
}

\bib{AGY}{article}{
      author={Alam, Mahbub},
      author={Ghosh, Anish},
      author={Yu, Shucheng},
       title={Quantitative {D}iophantine approximation with congruence
  conditions},
        date={2021},
        ISSN={1246-7405},
     journal={J. Th\'{e}or. Nombres Bordeaux},
      volume={33},
      number={1},
       pages={261\ndash 271},
         url={http://jtnb.cedram.org/item?id=JTNB_2021__33_1_261_0},
      review={\MR{4312709}},
}

\bib{JA}{incollection}{
      author={Athreya, Jayadev~S.},
       title={Random affine lattices},
        date={2015},
   booktitle={Geometry, groups and dynamics},
      series={Contemp. Math.},
      volume={639},
   publisher={Amer. Math. Soc., Providence, RI},
       pages={169\ndash 174},
         url={https://doi.org/10.1090/conm/639/12793},
      review={\MR{3379825}},
}

\bib{BV}{article}{
      author={Beresnevich, Victor},
      author={Velani, Sanju},
       title={An inhomogeneous transference principle and {D}iophantine
  approximation},
        date={2010},
        ISSN={0024-6115},
     journal={Proc. Lond. Math. Soc. (3)},
      volume={101},
      number={3},
       pages={821\ndash 851},
         url={https://doi.org/10.1112/plms/pdq002},
      review={\MR{2734962}},
}

\bib{BG3}{article}{
      author={Bj\"{o}rklund, Michael},
      author={Einsiedler, Manfred},
      author={Gorodnik, Alexander},
       title={Quantitative multiple mixing},
        date={2020},
        ISSN={1435-9855},
     journal={J. Eur. Math. Soc. (JEMS)},
      volume={22},
      number={5},
       pages={1475\ndash 1529},
         url={https://doi.org/10.4171/jems/949},
      review={\MR{4081727}},
}

\bib{BG}{article}{
      author={Bj\"{o}rklund, Michael},
      author={Gorodnik, Alexander},
       title={Central limit theorems for {D}iophantine approximants},
        date={2019},
        ISSN={0025-5831},
     journal={Math. Ann.},
      volume={374},
      number={3-4},
       pages={1371\ndash 1437},
         url={https://doi.org/10.1007/s00208-019-01828-1},
      review={\MR{3985114}},
}

\bib{BL}{article}{
      author={Bugeaud, Yann},
      author={Laurent, Michel},
       title={On exponents of homogeneous and inhomogeneous {D}iophantine
  approximation},
        date={2005},
        ISSN={1609-3321},
     journal={Mosc. Math. J.},
      volume={5},
      number={4},
       pages={747\ndash 766, 972},
         url={https://doi.org/10.17323/1609-4514-2005-5-4-747-766},
      review={\MR{2266457}},
}

\bib{CGGMS}{article}{
      author={Chow, S.},
      author={Ghosh, A.},
      author={Guan, L.},
      author={Marnat, A.},
      author={Simmons, D.},
       title={Diophantine transference inequalities: weighted, inhomogeneous,
  and intermediate exponents},
        date={2020},
     journal={Annali della Scuola Normale Superiore di Pisa, Classe di
  Scienze},
      volume={XXI},
}

\bib{DFV}{article}{
      author={Dolgopyat, Dmitry},
      author={Fayad, Bassam},
      author={Vinogradov, Ilya},
       title={Central limit theorems for simultaneous {D}iophantine
  approximations},
        date={2017},
        ISSN={2429-7100},
     journal={J. \'{E}c. polytech. Math.},
      volume={4},
       pages={1\ndash 36},
         url={https://doi.org/10.5802/jep.37},
      review={\MR{3583273}},
}

\bib{SE}{article}{
      author={Edwards, Sam},
       title={The rate of mixing for diagonal flows on spaces of affine
  lattices},
        date={2013},
}

\bib{EW}{book}{
      author={Einsiedler, Manfred},
      author={Ward, Thomas},
       title={Ergodic theory with a view towards number theory},
      series={Graduate Texts in Mathematics},
   publisher={Springer-Verlag London, Ltd., London},
        date={2011},
      volume={259},
        ISBN={978-0-85729-020-5},
         url={https://doi.org/10.1007/978-0-85729-021-2},
      review={\MR{2723325}},
}

\bib{BMV}{article}{
      author={El-Baz, Daniel},
      author={Marklof, Jens},
      author={Vinogradov, Ilya},
       title={The distribution of directions in an affine lattice: two-point
  correlations and mixed moments},
        date={2015},
        ISSN={1073-7928},
     journal={Int. Math. Res. Not. IMRN},
      number={5},
       pages={1371\ndash 1400},
         url={https://doi.org/10.1093/imrn/rnt258},
      review={\MR{3340360}},
}

\bib{EM}{article}{
      author={Eskin, Alex},
      author={Margulis, Gregory},
      author={Mozes, Shahar},
       title={Upper bounds and asymptotics in a quantitative version of the
  {O}ppenheim conjecture},
        date={1998},
        ISSN={0003-486X},
     journal={Ann. of Math. (2)},
      volume={147},
      number={1},
       pages={93\ndash 141},
         url={https://doi.org/10.2307/120984},
      review={\MR{1609447}},
}

\bib{GH}{article}{
      author={Ghosh, Anish},
      author={Han, Jiyoung},
       title={Values of inhomogeneous forms at {$S$}-integral points},
        date={2022},
        ISSN={0025-5793},
     journal={Mathematika},
      volume={68},
      number={2},
       pages={565\ndash 593},
         url={https://doi.org/10.1112/mtk.12137},
      review={\MR{4418458}},
}

\bib{GKY}{article}{
      author={Ghosh, Anish},
      author={Kelmer, Dubi},
      author={Yu, Shucheng},
       title={Effective density for inhomogeneous quadratic forms {I}:
  {G}eneric forms and fixed shifts},
        date={2022},
        ISSN={1073-7928},
     journal={Int. Math. Res. Not. IMRN},
      number={6},
       pages={4682\ndash 4719},
         url={https://doi.org/10.1093/imrn/rnaa206},
      review={\MR{4391899}},
}

\bib{Harman}{book}{
      author={Harman, Glyn},
       title={Metric number theory},
      series={London Mathematical Society Monographs. New Series},
   publisher={The Clarendon Press, Oxford University Press, New York},
        date={1998},
      volume={18},
        ISBN={0-19-850083-1},
      review={\MR{1672558}},
}

\bib{HS}{article}{
      author={Hartman, S.},
      author={Sz\"{u}sz, P.},
       title={On congruence classes of denominators of convergents},
        date={1960},
        ISSN={0065-1036},
     journal={Acta Arith.},
      volume={6},
       pages={179\ndash 184},
         url={https://doi.org/10.4064/aa-6-2-179-184},
      review={\MR{117189}},
}

\bib{KS}{article}{
      author={Katok, Anatole},
      author={Spatzier, Ralf~J.},
       title={First cohomology of {A}nosov actions of higher rank abelian
  groups and applications to rigidity},
        date={1994},
        ISSN={0073-8301},
     journal={Inst. Hautes \'{E}tudes Sci. Publ. Math.},
      number={79},
       pages={131\ndash 156},
         url={http://www.numdam.org/item?id=PMIHES_1994__79__131_0},
      review={\MR{1307298}},
}

\bib{KM2}{incollection}{
      author={Kleinbock, D.~Y.},
      author={Margulis, G.~A.},
       title={Bounded orbits of nonquasiunipotent flows on homogeneous spaces},
        date={1996},
   booktitle={Sina\u{\i}'s {M}oscow {S}eminar on {D}ynamical {S}ystems},
      series={Amer. Math. Soc. Transl. Ser. 2},
      volume={171},
   publisher={Amer. Math. Soc., Providence, RI},
       pages={141\ndash 172},
         url={https://doi.org/10.1090/trans2/171/11},
      review={\MR{1359098}},
}

\bib{KM1}{incollection}{
      author={Kleinbock, D.~Y.},
      author={Margulis, G.~A.},
       title={On effective equidistribution of expanding translates of certain
  orbits in the space of lattices},
        date={2012},
   booktitle={Number theory, analysis and geometry},
   publisher={Springer, New York},
       pages={385\ndash 396},
         url={https://doi.org/10.1007/978-1-4614-1260-1_18},
      review={\MR{2867926}},
}

\bib{NRS}{article}{
      author={Nesharim, Erez},
      author={R\"{u}hr, Rene},
      author={Shi, Ronggang},
       title={Metric diophantine approximation with congruence conditions},
        date={2020},
     journal={International Journal of Number Theory},
      volume={16},
      number={09},
       pages={1923\ndash 1933},
}

\bib{S60}{article}{
      author={Schmidt, Wolfgang~M.},
       title={A metrical theorem in diophantine approximation},
        date={1960},
     journal={Canadian J. Math.},
      volume={12},
}

\bib{SW}{article}{
      author={Schmidt, Wolfgang~M.},
       title={Asymptotic formulae for point lattices of bounded determinant and
  subspaces of bounded height},
        date={1968},
        ISSN={0012-7094},
     journal={Duke Math. J.},
      volume={35},
       pages={327\ndash 339},
         url={http://projecteuclid.org/euclid.dmj/1077377618},
      review={\MR{224562}},
}

\bib{sie}{article}{
      author={Siegel, Carl~Ludwig},
       title={A mean value theorem in geometry of numbers},
        date={1945},
        ISSN={0003-486X},
     journal={Ann. of Math. (2)},
      volume={46},
       pages={340\ndash 347},
         url={https://doi.org/10.2307/1969027},
      review={\MR{12093}},
}

\bib{Szusz58}{article}{
      author={Sz\"{u}sz, P.},
       title={\"{U}ber die metrische {T}heorie der diophantischen
  {A}pproximation.},
        date={1958},
     journal={Acta. Math. Sci. Hungar.},
      volume={9},
       pages={177\ndash 193},
}

\bib{Szusz}{article}{
      author={Sz\"{u}sz, P.},
       title={\"{U}ber die metrische {T}heorie der diophantischen
  {A}pproximation. {II}},
        date={1962/63},
        ISSN={0065-1036},
     journal={Acta Arith.},
      volume={8},
       pages={225\ndash 241},
         url={https://doi.org/10.4064/aa-8-2-225-241},
      review={\MR{153622}},
}

\end{biblist}
\end{bibdiv}

\end{document}